\documentclass[10pt]{amsart}
\usepackage{amssymb,amsmath,amsthm,amsfonts,color}
\usepackage{mathrsfs,dsfont,a4wide}

\theoremstyle{plain}
\newtheorem{theorem}{Theorem}[section]
\newtheorem{proposition}[theorem]{Proposition}
\newtheorem{lemma}[theorem]{Lemma}

\newtheorem{definition}[theorem]{Definition}
\theoremstyle{definition}
\newtheorem{remark}[theorem]{Remark}

\numberwithin{equation}{section}

\newcommand{\as}{{\mathcal A}}
\newcommand{\hs}{{\mathcal H}}

\newcommand{\fs}{{\mathcal F}}

\newcommand{\leb}{{\mathcal L}}
\newcommand{\ds}{{\mathcal D}}

\newcommand{\ms}{{\mathcal M}}

\newcommand{\bs}{{\mathcal B}}

\newcommand{\es}{{\mathcal E}}

\newcommand{\qs}{{\mathcal Q}}

\newcommand{\zs}{{\mathcal Z}}

\newcommand{\ys}{{\mathcal Y}}
\newcommand{\xs}{\boldsymbol{\mathcal X}}

\newcommand{\R}{{\mathbb R}}
\newcommand{\C}{{\mathbb C}}

\newcommand{\N}{{\mathbb N}}

\newcommand{\Z}{{\mathbb Z}}
\newcommand\II{{\mathbb I}}

\newcommand{\e}{\varepsilon}

\newcommand{\Om}{\Omega}
\newcommand{\Omb}{\overline{\Omega}}

\newcommand{\weakst}{\stackrel{\ast}{\rightharpoonup}}

\newcommand{\weakdue}{\stackrel{w\text{-}2}{\rightharpoonup}}
\newcommand{\weakduest}{\stackrel{w^*\text{-}2}{\rightharpoonup}}

\newcommand{\weak}{\rightharpoonup}
\newcommand{\wlystar}{$\text{weakly}^*$\;}

\newcommand{\wstar}{$\text{weak}^*$\;}
\newcommand{\duew}{$\text{two-scale weakly}$\;}
\newcommand{\duewst}{$\text{two-scale weakly*}$\;}

\newcommand{\hn}{\hs^{N-1}}

\newcommand{\pscal}[2]{\langle #1, #2 \rangle}

\newcommand{\eps}{\varepsilon}

\newcommand{\ph}{\varphi}

\newcommand{\bel}[1]{\begin{equation}\label{#1}}
\newcommand{\ee}{\end{equation}}

\newcommand{\pa}{\partial}

\renewcommand{\Omega}{\varOmega}
\renewcommand{\Gamma}{\varGamma}

\newcommand{\xoe}{\frac x\eps}
\newcommand{\ashom}{{\mathcal A}_{2s}^{hom}}
\newcommand{\qshom}{\qs_{2s}^{hom}}
\newcommand{\hshom}{\hs_{2s}^{hom}}
\newcommand{\eshom}{\es_{2s}^{hom}}

\newcommand\ootimes{\stackrel{\text{\it \tiny gen.}}{\otimes}}

\definecolor{verde}{RGB}{50,150,80}

\newcommand{\EEE}{\color{black}}

\title [Homogenization and Hencky plasticity]{Periodic homogenization in scalar Hencky plasticity} 
\author[G. Francfort] {Gilles Francfort} 
\address[Gilles Francfort]{Flatiron Institute}
\email[G. Francfort]{gilles.francfort@flatironinstitute.org}
\author[A. Giacomini] {Alessandro Giacomini}
\address[A. Giacomini]{DICATAM, Sezione di Matematica, Universit\`a di Brescia, 
Via Branze 43, 25123,  Brescia, Italy}
\email[A. Giacomini]{alessandro.giacomini@unibs.it}
\begin{document} 
\vskip .2truecm
\begin{abstract}
\small{This paper investigates the periodic homogenization of a two-phase elasto-plastic material in the framework of $\Gamma$-convergence. It  establishes the link between the homogenized functional and that which was previously obtained in \cite{FGhom} in the context of two-scale convergence. Although there is no gap between the two formulations and that involving two-scale has the hallmark of a (two-scale) elasto-plastic problem, we are unable to positively ascertain the elasto-plastic character of the homogenized energy except in a one-dimensional setting.
}
\end{abstract}
\maketitle
{\small \tableofcontents}

\section{Introduction}
\subsection{Introductory remarks}
In a previous work \cite{FGhom}, we investigated the periodic homogenization of a multi-phase elasto-plastic composite, each phase being endowed with an   arbitrary yield surface and elasticity. The only restriction was that the interfaces between the phases be piecewise $C^1$. 

Our main goal was to unravel the interaction between the evolution and the elasto-plastic microstructure; the reason was that the large body of engineering literature concerned with the macroscopic elasto-plastic behavior of heterogeneous materials operated exclusively under a priori assumptions on the allowed microscopic deformation mechanisms while the few mathematical works assumed strain hardening,  a regularizing mechanism under which the homogenization procedure becomes much simpler; see the references in the introduction of \cite{FGhom}.

The analysis was conducted using mainly the framework of two-scale convergence because we were unable to depart from the periodic setting. The resulting evolution is a rather complex two-scale evolution which cannot be easily reduced to an evolution solely at the  macroscopic scale. In particular   it is far from obvious whether homogenized elastic moduli and a homogenized dissipation potential can be associated with the macroscopic evolution, so   the label {\it elasto-plasticity} might be an inadequate qualification for that evolution.

\par
In this  paper, we propose to delve into this labeling issue by trying to tie the two-scale description to a purely macroscopic (one-scale) description.

Unfortunately, doing so in the original framework of \cite{FGhom} would be an overwhelming undertaking at present. This is so for two distinct reasons. First,  in the absence of a clear definition of  the macroscopic elastic and dissipation potentials, it is not possible to provide a rational framework for time-dependent evolutions. Consequently, 
we abandon the evolution aspect of the problem and zero in on what could be seen as a static version of elasto-plasticity, Hencky plasticity in the  parlance of solid mechanics. 
Then, if dealing with the problem in the vectorial case, we run into the issue of approximating  admissible configurations involving 
deviatoric plastic measures -- matrix-valued measures with $0$-trace -- while preserving the intricate kinematic structure of elasto-plasticity;  for example, localization procedures  employed in the analysis produce extra strains lacking   either the deviatoric character, or the required summability properties. This issue, whether it can be dealt with or not,  is a byproduct of the particulars of elasto-plastic kinematics and, as such,  it is  collateral to the main focus of our efforts 
which  is why we  abandon the vectorial case and restrict those to that of a scalar-valued ``displacement" field. 

Accordingly, we will decompose the gradient of the displacement $u$ into an ``elastic" field $e$ and a ``plastic" field  $p$, both $\R^N$-valued; that decomposition reads as
$$
Du=e+p.
$$
 Throughout we will use the  terminology of elasto-plasticity for denominational convenience. A mechanics oriented person would label our simplified setting  an anti-plane one.

With those caveats in mind, we establish the following results. 

Consider the two-scale homogenization process which gives rise to a two-scale description involving a functional defined over triplets $(u,E,P)$   where   $u$   is a  scalar-valued displacement (a $BV$-function), $E$  a two-scale elastic strain (an $L^2$-function depending on both the macroscopic variable $x$ -- that which spans  the physical domain-- and the microscopic variable $y\in \ys$ -- that which spans the microstructure --), and $P$  a two-scale plastic measure (a measure in both $x$ and $y$). Then, minimizing that functional over all pairs $E,P$ which are associated to the same triplet $(u,e,p)$ -- with $e,p$ the averages of $E,P$ over the microscopic variable -- through a rather subtle kinematic compatibility condition that we recall in Subsection \ref{sub:2sc-res}, we recover the homogenized energy  $\es^{hom}(u,e,p)$  of the elasto-plastic problem (see Theorem \ref{thm:hom-res}). 
The proof of that result is rather long and technical and will occupy the bulk of the paper.   It demonstrates that there is no gap, in the periodic setting, between the two-scale  and the one-scale viewpoints.
\par
 Our result is  of a similar vein to  that of \cite{FF} which  links 
 periodic homogenization of a functional with linear growth to two-scale convergence of measures. While continuity of the functional in the space variable is assumed in \cite{FF}, our setting encompasses  a natural  requirement, that of dealing with  a two-phase microstructure. This necessitates discontinuous coefficients.  Further,  we also need  to incorporate the somewhat intricate kinematic constraints of elasto-plasticity (see Remark \ref{rmk:FF}  for further details).
\par
 Information about the possible elasto-plastic structure of $\es^{hom}(u,e,p)$ can be suggested by the associated one-field  functional $\fs^{hom}(u)$ obtained by minimizing $\es^{hom}$ over all compatible triplets $(u,e,p)$ at fixed $u$. 
$\fs^{hom}$ is actually the homogenized functional for the one-field description $\fs_\e(u)$ associated with the original periodic elasto-plastic problem (see Theorem  \ref{thm:FshomEshom}). In the one-field setting, the functional to be homogenized is convex and with linear growth on $BV$; $\fs^{hom}$ can thus be described in terms of an energy density determined by a cell formula as first derived in \cite{bouchitte86}. 
\par
Using that cell formula we demonstrate that,  in one dimension, the homogenized behavior is indeed  elasto-plastic,  because it is the inf-convolution of a quadratic elastic energy with a positively one-homogeneous plastic potential. In higher dimensions, we can only produce energetic bounds on the homogenized energy, and this under some additional assumptions. The bounding energies are  indeed elasto-plastic, but it is unclear whether this is true of  the actual homogenized energy. 

 The use of the one-field homogenized functional in investigating the elasto-plastic character of the homogenized functional reveals our current failure in exploiting the link between two-scale homogenization and its one-scale analogue provided by Theorem \ref{thm:hom-res}. This is so because we lack a cell formula in that framework. However, the appeal to  a cell formula in the one-field framework fails to adjudicate the question at hand.   We should note that a negative result would deliver   a decisive blow \EEE to a large body of engineering literature where the goal is to discover the macroscopic parameters that could characterize the macroscopic elasto-plastic behavior because it would become impossible to define what the plastic strain is at the macroscopic level. There is however a two-scale plastic strain as we demonstrated in \cite{FGhom} and forcefully confirmed through Theorem \ref{thm:hom-res} of this contribution.

\par
The paper is organised as follows. 
\par
In Section \ref{sec:Hencky-het}, we introduce, in the case of a heterogeneous material made of two phases (the only setting we will address in this work),  the elasto-plastic formulation,  a three-field formulation in terms of the displacement field $u$, the elastic strain $e$, and the plastic strain $p$, the latter two being $\R^N$-vectors in this scalar setting and then
 set up the periodic problem.

Section \ref{sec:twoscale} has a first subsection
devoted to a brief recall of the relevant results in \cite{FGhom} while its second subsection goes to the heart of the argument. It is a kind of two-scale $\Gamma$-limit process which involves the two-scale homogenized energies arising out of the homogenization results of \cite{FGhom}. The $\Gamma\text{-}\liminf$-part of the argument is immediate in view of the previously obtained results recalled in Subsection  \ref{sub:2sc-res}.
The effort lies in establishing the $\Gamma\text{-}\limsup$ which leads to Theorem \ref{thm:2sc-gl}. Then the final subsection establishes the homogenization result  (Theorem \ref{thm:hom-res}).  In the approximation arguments, we will need to deal with distributions and measures on the periodic torus: we collect  the relevant results in the Appendix.

Section \ref{sec:notplastic} introduces the one-field formulation  and its homogenization and  demonstrates that  minimizing the two-scale homogenized kinematics over all two-scale elastic and plastic strains permits to recover the homogenized one-field functional, at least in the admittedly restrictive periodic setting. This is the object of Theorem \ref{thm:FshomEshom}. 
It then  addresses  in Subsection \ref{sub:cell-form}  the issue of whether the  homogenized formulation retains an elasto-plastic character. This we only know how to do by considering the cell formula classically associated with the one-field homogenized energy and attempting to discover the underlying elasto-plastic structure. We show that it is indeed so in the one dimensional case (see Paragraph \ref{subsub:1d}). In the multi-dimensional case (Paragraph \ref{subsub:bounds})  we can only produce bounds on the homogenized energy; those bounds do retain an elasto-plastic character (see Proposition \ref{prop:bds-Whom}) at least in an isotropic setting and with a proportionality restriction between the elastic coefficients and the yield stresses. Whether the homogenized energy itself always retains that property remains open at present.

Finally, the reader will undoubtedly note that force loads are not considered in this work. As  explained in 
\cite[Remark 2.9]{francfort.giacomini}, this is no restriction,  provided that a uniform safe load condition  is satisfied; for details refer to that remark in \cite{francfort.giacomini}.

\subsection{Notation}
\label{sec:prel}

The following notation will be adopted throughout.

\vskip10pt\noindent{\bf  General notation.}
 $\II$ denotes the identity matrix on $\R^N$. For $A\subseteq\R^N$, $\chi_A$ denotes the characteristic function of $A$, {\it i.e.}, $\chi_A(x)=1$ for $x\in A$ and $\chi_A(x)=0$ for $x\not\in A$. 
 The indicator function of $A$, denoted by $\mathbf I_A$, is defined as $\mathbf I_A(x)=0$ for $x\in A$, and $\mathbf I_A(x)=+\infty$ for $x\not\in A$. 
  The symbol $\lfloor_A$ stands  for ``restricted to $A$". $B(0,r)$ and ${\overline B}(0,r)$ respectively stand for the open and closed balls of radius $r$ centered at $0$.
 Finally $\leb^N$  {stands} for the usual Lebesgue measure, while $\hn$  {denotes} the $(N-1)$ dimensional Hausdorff measure.

\vskip15pt\noindent{\bf  Measures.}
 If $E$ is a locally compact separable metric space, and $X$ a finite dimensional normed space, 
$\ms_b(E;X)$ will denote the space of finite Radon measures on $E$ with values in $X$. 
For $\mu\in \ms_b(E;X)$, we denote by $|\mu|$ its total variation and by $\displaystyle\frac{\mu}{|\mu|}$  the Radon-Nikodym derivative of $\mu$ with respect to $|\mu|$. 
The space $\ms_b(E;X)$ is the topological dual of $C^0_0(E;X^*)$, the set of continuous functions $u$ from $E$ to 
the vector dual $X^*$ of $X$ which ``vanish at the boundary", {\it i.e.}, such that for every $\e>0$ there exists 
a compact set $K\subseteq E$ with $|u(x)|<\e$ for $x\not\in {K}$;  $\mu_n\weakst \mu$ will denote the associated weak* convergence. Finally we will denote by $\ms_b^+(E)$ the space of 
positive bounded Radon measures on $E$.  
\par

\vskip10pt
We will use generalized product and disintegration of measures, for which we refer the reader to \cite[Section 2.5]{ambrosio.fusco.pallara}. Given $E,F$ locally compact separable metric spaces, and $\eta\in \ms_b^+(E)$, a map $x\mapsto \mu_x\in \ms_b(F)$ is said to be $\eta$-measurable if the map
$$
x\mapsto \mu_x(B)
$$
is $\eta$-measurable for every Borel set $B\subseteq F$. Assuming moreover that
the map $x\mapsto |\mu_x|(F)$ is $\eta$-summable, the generalized product $\eta\ootimes \mu_x\in \ms_b(E\times F)$ is defined through the equality
$$
\pscal{\eta\ootimes \mu_x}{f}:=\int_E \left( \int_F f(x,y)\,d\mu_x(y)\right)\,d\eta(x),
\qquad f\in C^0_0(E\times F).
$$

\par
Moreover (see \cite[Theorem 2.28]{ambrosio.fusco.pallara}), every $\mu\in \ms_b(E\times F)$ can be disintegrated, {\it i.e.}, it can written as a generalized product $\eta\ootimes \mu_x$. Here $\eta$ is the push forward of $|\mu|$ along the projection on $E$, {{\it i.e.}}, for every Borel set $B\subseteq E$
$$
\eta(B):=|\mu|(B\times F),
$$
while $x\mapsto \mu_x\in \ms_b(F)$ is a suitable $\eta$-measurable map.
The generalized product technique, and the associated disintegration result, are easily extended to the case of vector valued finite Radon measure.

\vskip10pt
\noindent{\bf  Functions of bounded variation.}
{Let $\Om\subseteq \R^N$ be an open set.}
In this paper  the ``displacement field" $u$ lies in $BV(\Om)$, the space of functions with bounded variations. We refer the reader to e.g.\;\cite{ambrosio.fusco.pallara} for background material. \par
We say that
 $
u_n\weakst u\qquad\text{\wlystar in }BV(\Om)
$
iff
 $u_n\to u\text{ strongly in $L^1(\Om;\R)$ and } Du_n\weakst Du \;\text{ \wlystar in }\ms_b(\Om;\R^N).$ 
Bounded sequences in $BV(\Om)$ always admit a \wlystar converging subsequence.

\vskip15pt\noindent{ \bf Periodic torus.}
Our analysis of the homogenization problem relies on an extensive use of two-scale convergence (see Section \ref{sec:twoscale}). We thus need to consider the space of $\Z^N$-periodic continuous (or $C^1$) functions on $\R^N$, and its dual, a space of measures that enjoys suitable periodicity properties. These spaces are most conveniently viewed as acting on a torus.
\par
Let $\ys:=\R^N/\Z^N$ be the $N$-dimensional torus, which we endow with the standard differentiable structure as a quotient space. Let us denote with $\pi_\ys$ the projection of $\R^N$ onto $\ys$.
\par
 Smooth functions on $\ys$ are naturally identified with smooth functions on $\R^N$ which are $\Z^N$-periodic. For $\varphi$ function on $\ys$, we will denote with $\varphi_{\R^N}$ the associated periodic function on $\R^N$ and conversely, given a periodic function $\psi$ on $\R^N$, we will write $\psi_\ys$ for the associated function on $\ys$. Taking advantage of this equivalence we are at liberty to define differential operators for functions or vector fields on $\ys$: for example, the gradient of a function on $\ys$ is 
$$
D_y \varphi= (D_x\varphi_{\R^N})_\ys.
$$
This definition is consistent with that of gradient for the standard Riemannian structure on $\ys$ inherited by $\R^N$ as a differentiable manifold. 
\par
 Distributions or  measures on $\ys$ are also  associated to "periodic" distributions or measures on $\R^N$  allowing us to   transfer the standard euclidean techniques to the torus and therefore to keep a complete formal analogy. In Sections \ref{sec:twoscale} and \ref{sec:notplastic} we will 
need in particular the regularization by convolution of measures. The main definitions and relevant properties are collected in the Appendix  which can be viewed as a self-contained elementary introduction to convolution of measures on the torus.
\par

 For any $\e>0$ and $\zs\subset\ys$, we define
\begin{equation}
\label{eq:def-zeps}
\zs_\e:=\{x\in \R^N:\displaystyle \pi_\ys(x/\e)\in \zs\},
\end{equation}
while for any function $F:\ys\to X$, where $X$ is some set, we will write simply $F\left(x/\e\right)$ to denote the $\e \Z^N$-periodic function $F(\pi_\ys(x/\e))$.

\medskip

Finally, as alluded to before, we adopt throughout the terminology of elasto-plasticity, using  words like displacement, stress, elastic and plastic strains, yield stress, elasticity tensor, etc..., although we do not tackle the complexity of the true elasto-plastic problem, and most notably the requirement that the plastic strain should be deviatoric. 

\section{ Static heterogeneous scalar Hencky plasticity }
\label{sec:Hencky-het}
In this section we adapt to the scalar setting the approach to heterogeneous plasticity developed in \cite{francfort.giacomini},  restricting our attention to two phase materials. The relevant definitions are the object of the first subsection while the periodic setting is detailed in the second subsection.
\subsection{Two-phase heterogeneity}\label{sub:2phase}
Let  $\Om\subset \R^N$ be an open, bounded set with (at least) Lipschitz boundary {and exterior normal $\nu$}. Further, let the Dirichlet part $\Gamma_d$ of $\partial\Om$  be a non empty open set in the relative topology of $\partial\Om$. 
\par 

\vskip10pt\noindent{\sf Admissible configurations.} 
Given the boundary displacement $w\in H^1(\R^N)$, the family of admissible configurations relative to $w$, is the set of triplets $(u,e,p)$ with
$$
u\in BV(\Om),\qquad e\in L^2(\Om;\R^N),\qquad p\in \ms_b(\Om\cup \Gamma_d;\R^N),
$$
and such that
\begin{equation*}
Du=e+p\quad\text{ in }\Om,\qquad p= (w-u)\nu \,\hn\lfloor \Gamma_d\quad\text{ on }\Gamma_d,
\end{equation*}
where $(w-u)$ is to be understood in the sense of traces.
\par
 In order to handle the Dirichlet boundary condition, it proves convenient to consider $\Om'\subseteq \R^N$ open bounded with $\Om\subset\Om'$ and such that $\partial \Om\cap \Om'=\Gamma_d$. Given a boundary displacement $w\in H^1(\R^N)$, 
and a configuration $(u,e,p)\in \as(w)$, we may extend $u,e,p$ to $\Om'$ by setting
\begin{equation}
\label{eq:wOm'}
u=w,\qquad e=\nabla w,\qquad p=0\qquad\text{on }\Om'\setminus\Omb.
\end{equation}
It is readily checked that 
\begin{equation}
\label{eq:comp-Om'}
Du=e+p\qquad \text{on }\Om'.
\end{equation}
Then the family of admissible configurations for $w$ can be described  equivalently  as
\begin{multline}
\label{eq:asOm'}
\as(w,\Om')=\{(u,e,p)\in BV(\Om')\times L^2(\Om';\R^N)\times \ms_b(\Om';\R^N)\;:\,\\
 \text{\eqref{eq:wOm'} and \eqref{eq:comp-Om'} are satisfied}\}.
\end{multline}

\par

\vskip10pt\noindent{\sf Two-phase elasto-plastic materials.}  
We assume that the domain $\Om$ is made up of two open phases $\Om_1,\Om_2$, which are Lipschitz  domains,  together with their interface  $\Sigma:=\pa \Om_1\cap\pa \Om_2\cap\Om$.
\par
The elasto-plastic properties of $\Om$ are given in terms of an  elasticity tensor and a  dissipation potential.
\begin{itemize}
\item The elasticity tensor is  given by a symmetric $N\times N$-matrix of the form $\C(x)$
with $\C:=\C_i $  on  $\Om_i$ and  
\bel{eq:Cinc}
\C_i  \text{ is a symmetric definite positive (constant) matrix.}
\ee
\vskip5pt
\item In elasto-plasticity, the deviatoric part of the stress $\sigma=\C e$ is  restricted by the yield condition. When specialized to a scalar two-phase setting, we are simply led to assuming the existence of  a convex compact set $K_i\subset\R^N$ for  phase $\Om_i, i=1,2$. We further assume that those cannot be too small or too large, {\it i.e.}, there exist $c_1,c_2>0$ such that
\bel{eq:Kinc}
B(0, c_1)\subset K_i  (i=1,2)\subset B(0,c_2).
\ee
\par
In the rest of the paper, we will further assume that 
\bel{eq:ordK}
K_1\subset K_2.
\ee
 The ordering assumption \eqref{eq:ordK} is not present in \cite{francfort.giacomini, FGhom}. It   simplifies the proof of our main homogenization result, that is Theorem \ref{thm:2sc-gl}.  Absent this restriction,  additional technical difficulties arise which we prefer to avoid, even if it is our unsubstantiated belief that  the general strategy of the proof would not be altered. 
\vskip5pt
\item For all $x\in\Om_i$ and $\xi\in \R^N$, we define the dissipation potential
to be
$$
H_i(\xi):= \sup\{\tau\cdot \xi: \tau\in K_i\}.
$$
Thanks to \eqref{eq:ordK} we get
\begin{equation*}
H_1\le H_2.
\end{equation*}
The dissipation potential $H: (\Om \cup \Gamma_d) \times\R^N \to [0,+\infty)$ is defined as
\bel{eq:defHdens}
H(x,\xi)=
\begin{cases}
H_{1}(\xi)&\text{if }x\in\Om_1\cup \Sigma \cup (\pa\Om_1\cap \Gamma_d)\\
H_{2}(\xi)&\text{otherwise.}\\
\end{cases}
\ee
 The restriction \eqref{eq:ordK} underlies the definition \eqref{eq:defHdens} of $H$;  since $H=H_1$ (the lowest dissipation) on the interface $\Sigma$,  $H(\cdot, \xi)$ is  lower semicontinuous on $\Om \cup \Gamma_d$.  Moreover $H(x,\cdot)$  is convex and positively one-homogeneous and satisfies for every $x\in \Om\cup \Gamma_d$ and $\xi\in\R^N$
$$
c_3|\xi|\le H(x,\xi)\le c_4|\xi|,
$$
where $c_3, c_4>0$. 
\end{itemize}
For every  $(u,e,p)\in \as(w,\Om')$  we define the elastic energy as
$$
\qs(e):=\frac{1}{2}\int_{\Om'} \C(x) e\cdot e\,dx,
$$
where $\C$  has been extended to an arbitrary (possibly $x$-dependent) elasticity tensor on $\Om'\setminus \Om$ . Note that
the additional contribution to the elastic energy on $\Om'\setminus \Om$ is  given by  $1/2\int_{\Om'\setminus \Om}\C \nabla w\cdot \nabla w\ dx$; it is a fixed contribution that only depends  on $w$. Similarly we define 
 the dissipation functional to be
$$
\hs(p):=\int_{\Om\cup\Gamma_d}H\left(x,\frac{p}{|p|}\right)\,d|p|=\int_{\Om'}H\left(x,\frac{p}{|p|}\right)\,d|p|
$$
and set 
\begin{equation}
\label{eq:estot}
\es(u,e,p):=\qs(e)+\hs(p),
\end{equation}
which we call  the {\it total energy} of the configuration. 
\par
 The {\it  elasto-plastic equilibrium problem} is given by
\begin{equation}
\label{eq:minpb}
\min_{(u,e,p)\in \as(w,\Om')}\es(u,e,p).
\end{equation}

Existence of solutions to the previous problem is a simple consequence of the lower semicontinuity properties of $\qs$ and $\hs$ on $L^2(\Om;\R^N)$ and $\ms_b(\Om\cup \Gamma_d;\R^N)$,  respectively. Indeed, being quadratic, $\qs$ is lower semicontinuous with respect to the weak topology of $L^2(\Om;\R^N)$, while $\hs$ is  lower semicontinuous on $\ms_b(\Om\cup \Gamma_d;\R^N)$ with respect to the weak* convergence in view of Reshetnyak's lower semicontinuity theorem (see \cite[Theorem 2.38]{ambrosio.fusco.pallara}).
 Since the total energy $\es$ is coercive in $e$ and $p$, and consequently in $u$ in $BV(\Om')$ in view of \eqref{eq:wOm'} and \eqref{eq:comp-Om'}, 
\begin{equation}\label{eq:Kin}
\|u\|_{BV(\Om')}+\|e\|_{L^2(\Om;\R^N)}+\|p\|_{\ms_b(\Om';\R^N)}\le C(1+\es(u,e,p)),
\end{equation}
for some $C>0$ (depending on $w$). The existence of minimizers for \eqref{eq:minpb} follows by applying the Direct Method of the Calculus of Variations.

\begin{remark}
\label{rem:geom}
 The previous setting is a suitable simplification of the notion of  geometrically admissible multiphase domains described in \cite{francfort.giacomini}.  There  the ordering assumption \eqref{eq:ordK} is not assumed and a more careful definition of the dissipation on $\Sigma$ is required, forcing additional structure of the interface. Under the ordering assumption \eqref{eq:ordK},  no further structure is necessary. 
\hfill\P
\end{remark}

\begin{remark}[\bf Relaxation from smooth configurations]
\label{rem:relax}
A localization/translation/regular\-ization procedure similar to those of the lemmata in Subsection \ref{subsec:2s} below -- interpreting  the $y$-approximations as $x$-approximations -- would demonstrate that the total energy $\es$ in \eqref{eq:estot} is the lower semicontinuous envelope with respect to the product of the strong $L^1(\Om')$-topology, the weak $L^2(\Om';\R^N)$-topology and the weak*$ \ms_b(\Om';\R^N)$-topology of the functional
\bel{eq:esrel}
\es_{reg}(u,e,p)=\begin{cases}\qs(e)+\hs(p),&\text{if }(u,e,p)\in\as_{reg}(w, \Om')\\ +\infty&{\rm else,}\end{cases}
\ee
where 
\begin{multline*}
\as_{reg}(w, \Om'):=\left\{(u,e,p)\in \as(w, \Om')\,:\, u\in W^{1,1}(\Om'), \,\,p\in L^1(\Om';\R^N),\right.\\
\left.\text{$u=w$ and $p=0$ on $\Om'\setminus \Omb$}\right\}.
\end{multline*}
Note that, for $(u,e,p)\in \as_{reg}(w,\Om')$
$$
\hs(p)=\int_\Om H( x, p(x))\,dx,
$$
so that the values of the dissipation functional on the interface $\Sigma$ becomes irrelevant, as it has zero Lebesgue measure. As a consequence, the definition  $H=H_1$ on $\Sigma$  can be viewed as an outcome of a variational relaxation procedure.\EEE
\par
 The issue of relaxation from smooth configurations in plasticity has been addressed in \cite{mora} in the homogeneous case.  To the best of our knowledge an analogous relaxation process has not been performed in  a heterogeneous context, even in our case which is admittedly the simplest heterogeneity one can envision.  Let us finally remark that the definition of $H$ on $\Sigma$ has  also been motivated in \cite{francfort.giacomini} through a vanishing hardening approximation for quasi-static evolutions and this even in the  case without assumption \eqref{eq:ordK}. 
\hfill\P
\end{remark}

\EEE

\subsection{Periodic heterogeneous materials}
\label{sub:unit-cell}
We now detail the structure of scalar periodic heterogenous elasto-plastic materials. Elastic and plastic properties are defined through periodic functionals that we can assume defined through the use of the periodic torus $\ys$.
\par
Let the torus $\ys$ be  made of two disjoint open phases  with Lipschitz boundary  $\ys_1,\ys_2$, together with the interface $\Sigma$, {\it i.e.,} $\ys= \ys_1\cup  \ys_2\cup\Sigma$. 
\begin{itemize}
\item The elasticity tensor is  given by  a symmetric $N\times N$-matrix of the form $\C(y)$
with $\C:=\C_i $  on  $\ys_i$ and  
\bel{eq:Cpos2}\C_i  \text{ is a symmetric definite positive (constant) matrix.}
\ee

\item The set of admissible stresses is a convex compact  set $K_i\subset\R^N$ for each phase $\ys_i$ satisfying \eqref{eq:Kinc} and \eqref{eq:ordK}, The associated dissipation potential $H: \ys\times \R^N\to [0,+\infty[$ is given by
\bel{eq:defHper}
H(y,\xi)=\begin{cases}H_i(\xi) &\text{if }y\in\ys_i\\[2mm]H_1(\xi),&\text{if }y\in \Sigma. 
\end{cases}
\ee
As before, $H$ is lower semicontinuous since $H_1\le H_2$ (see \eqref{eq:ordK}) and  $H(y,\cdot)$  is convex and positively one-homogeneous in $\xi$, with
\begin{equation}
\label{eq:bdsHper}
c_3|\xi|\le H(y,\xi)\le c_4|\xi|,\mbox{ for all } y\in \ys
\end{equation}
where $c_3, c_4>0$.
\end{itemize}
Given $\e>0$, the domain $\Om$ is made up of the two phases $(\ys_1)_\e,(\ys_2)_\e$ {(see \eqref{eq:def-zeps})}. 
For  $(u,e,p)\in \as(w,\Om')$  defined in \eqref{eq:asOm'} we consider the elastic energy
\begin{equation}
\label{eq:defQe}
\qs_\e(e):=\frac{1}{2}\int_{\Om'} \C\left(\frac{x}{\e}\right) e\cdot e\,dx
\end{equation}
and the dissipation functional 
\begin{equation}
\label{eq:defHe}
\hs_\e(p):=\int_{\Om\cup\Gamma_d}H\left(\frac{x}{\e},\frac{p}{|p|}\right)\,d|p|=\int_{\Om'}H\left(\frac{x}{\e},\frac{p}{|p|}\right)\,d|p|.
\end{equation}
Here we use the fact that $\C(\cdot, \xi)$ and $H(\cdot,\xi)$ can be interpreted as periodic functions on $\R^N$. 
\par
The total energy of the configuration $(u,e,p)\in \as(w,\Om')$ is accordingly given by
\begin{equation}
\label{eq:totEeps}
\es_\e(u,e,p):=\qs_\e(e)+\hs_\e(p),
\end{equation}
and  the periodic elasto-plastic equilibrium problem  takes the form
$$
\min_{(u,e,p)\in \as(w,\Om')}\es_\e(u,e,p).
$$
Observe that inequality \eqref{eq:Kin} holds true for $\es_\e$ in lieu of $\es$ with a constant $C$ independent of $\e$. 
\par
We are interested in the  asymptotic behaviour as $\e\to 0^+$  of the energy $\es_\e$ and of the  minimizers $(u_\e,e_\e,p_\e)$  which trivially exist by lower semi-continuity and coercivity. Note that the contribution to the elastic energy of $\Om'\setminus \Om$ converges to $\frac{1}{2}\int_{\Om'\setminus \Om}\bar \C \nabla w\cdot\nabla w\,dx$ as $\e\to 0$, where $\bar \C:=\int_\ys \C(y)\,dy$. In other words, the asymptotic behavior on $\Om'\setminus \Om$ is trivial.
\section{ Homogenization for periodic scalar Hencky plasticity}
\label{sec:twoscale}
 
In this section we study the asymptotic behaviour of the functional $\es_\e$ defined in \eqref{eq:totEeps} as $\e$ goes to $0$. 
After recalling the definition of two-scale convergence for functions and measures in Subsection \ref{sub:2sc-res}, we describe in Subsection \ref{sub:Henky2s} the main results concerning two-scale periodic elasto-plasticity obtained in  \cite[Sections 4, 5]{FGhom} which we specialize to the simpler $BV$-setting. In Subsection \ref{subsec:2s} we prove that the two-scale total energy $\eshom$ can be seen as a suitable $\Gamma$-limit of the energies $\es_\e$ in the sense of two-scale convergence. Finally in Subsection \ref{sub:homHencky} we characterize the $\Gamma$-limit of $\es_\e$ in terms of $\eshom$ (see Theorem \ref{thm:hom-res}).

\subsection{ Two-scale convergence}
\label{sub:2sc-res}
 We refer the reader to \cite{allaire, nguetseng}  for the classical notion of two-scale convergence in Lebesgue spaces and to  \cite{amar} for that of two-scale weak* convergence of measures.

\begin{definition}[\bf Two-scale convergence for functions and measures]
\label{def:t-s-meas}
Let $\Om\subseteq\R^N$ be {an open} set.
\begin{itemize}

\item[(a)] Let $\{u_\e\}_{\e>0}$ be a  family in $L^p(\Om)$ for some $p\in [1,+\infty)$, and let $u_0\in L^p(\Om\times \ys)$. Then
$$
u_\e\weakdue u_0\qquad \text{two-scale weakly in }L^p(\Om\times\ys)
$$
if $\{u_\e\}_{\e>0}$ is bounded in $L^p(\Om)$ and if for every $\chi\in C^0_c(\Om\times\ys)$
$$
\lim_{\eps\to 0^+}\int_\Om u_\e(x)\chi\left(x,\frac{x}{\eps}\right)\,dx=\int_{\Om\times \ys}u(x,y)\chi(x,y)\,dxdy.
$$

\item[(b)]  Let $\{\mu_\eps\}_{\eps>0}$ be a family in $\ms_b(\Om)$ and consider $\mu_0\in \ms_b(\Om\times \ys)$. Then, 
$$
\mu_\eps\weakduest \mu_0\qquad\text{\duewst in }\ms_b(\Om\times \ys)
$$
iff,  for every $\chi\in C^0_0(\Om\times\ys)$,
$$
\lim_{\eps\to 0^+}\int_\Om \chi\left(x,\frac{x}{\eps}\right)\,d\mu_\eps(x)=\int_{\Om\times \ys}\chi(x,y)\,d\mu_0(x,y).
$$
\end{itemize}
\end{definition}

Bounded sequences in $L^p(\Om)$ for $p\in (1,+\infty)$ or in $\ms_b(\Om)$ always admit two-scale convergent subsequences (see \cite[Theorem 1.2]{allaire} for functions and \cite[Theorem 3.5]{amar} or Remark \ref{rem:ts1} below for measures).

\begin{remark}
\label{rem:ts1}
The family $\{\mu_\eps\}_{\eps>0}$ determines a family of measures $\{\lambda_\eps\}_{\eps>0}\subset \ms_b(\Om\times \ys)$ obtained by setting
$$
\int_{\Om\times \ys}\chi(x,y)\,d\lambda_\e(x,y):=\int_\Om \chi\left(x,\frac{x}{\eps}\right)\,d\mu_\eps(x)
$$
for every $\chi\in {C^0_0}(\Om\times \ys)$. Thus $\mu_0$ is simply the weak* limit in $\ms_b(\Om\times \ys)$ of $\{\lambda_\e\}_{\e>0}$.  In view of this identification, compactness for two-scale weak* convergence easily follows from that of standard weak* convergence for measures.
\hfill\P\end{remark}

\begin{remark}[\bf Weak convergence and two-scale weak convergence]
\label{rem:marginal}
The following properties of two-scale convergence   immediately follow upon taking test functions of the form $\chi(x,y):=\varphi(x)\psi(y)$ with $\varphi\in C^0_c(\Om)$ and $\psi\in C^0(\ys)$.
\begin{itemize}
\item[(a)]  If $u_\e \weakdue u_0(x,y)$ two-scale weakly in $L^p(\Om\times \ys)$ and $u_\e\weak u$ weakly in $L^p(\Om)$, then, for a.e. $x\in \Om$,
$$
u(x)=\int_\ys u_0(x,y)\,dy.
$$
\item[(b)] If $\mu_\e \weakduest \mu_0$ two-scale weakly* in $\ms_b(\Om\times \ys)$ and $\mu_\e\weakst \mu$ weakly* in $\ms_b(\Om)$, then, for every Borel set $B\subseteq \Om$,
$$
\mu(B)=\mu_0(B\times \ys).
$$
\end{itemize}
Items (a) and (b) show that weak or weak* limits can be recovered from their associated weak or weak* two-scale limits.
\hfill\P
\end{remark}

For our homogenization problem in scalar Hencky plasticity, we will need to  consider  two-scale weak* limits of measures which are also gradients of $BV$ functions. For $\Om\subseteq \R^N$ {open}, set
\begin{multline}
\label{defX}
\xs(\Om) :=\left\{\mu \in  \ms_b(\Om\times \ys): D_y\mu \in \ms_b(\Om\times \ys; \R^N), \right.\\
\left.{\mu(F\times \ys)=0\text{ for every Borel set $F\subseteq \Om$}}\right\},
\end{multline}
where $D_y\mu$ denotes the distributional  gradient of $\mu$ with respect to $y$ (see Appendix).  Following  \cite[Proposition 4.10]{FGhom}, if $u_\e\weakst u$ weakly* in $BV(\Om)$, then up to a subsequence
$$
Du_\e \weakduest Du\otimes \leb^N_y+D_y\mu\qquad\text{two-scale weakly* in }\ms_b(\Om\times \ys;\R^N)
$$
for some $\mu\in \xs(\Om)$. The following proposition enumerates the main properties of $\xs(\Om)$ that will be used in the sequel. It is the scalar reduction of \cite[Proposition 4.7]{FGhom}.

\begin{proposition}
\label{lem:Xs}
Let $\mu\in \xs(\Om)$. Then there exist $\eta\in \ms^+_b(\Om)$ and a Borel map
$(x,y)\in \Om\times \ys\mapsto \mu_x(y)\in \R$ such that, for $\eta$-a.e. $x\in \Om$,
\begin{equation}
\label{eq:mux-bd}
\mu_x\in BV(\ys),\qquad \int_\ys\mu_x(y)\,dy=0,\qquad {|D_y\mu_x|(\ys)\ne 0,}
\end{equation}
and
\bel{eq:dec-mu}
\mu=\mu_x(y)\,(\eta\otimes \leb^N_y).
\ee
Moreover, the map $x\mapsto D_y\mu_x\in \ms_b(\ys;\R^N)$ is $\eta$-measurable and
$$
D_y\mu=\eta\ootimes D_y\mu_x.
$$
\end{proposition}

\subsection{ Two-scale setting for scalar Hencky plasticity}
\label{sub:Henky2s}
We now define the set of admissible two-scale (kinematically admissible) configurations.

\begin{definition}[\bf Kinematically admissible two-scale configurations]
\label{set:Aw2} 
$\ashom(w,\Om')$, the family of admissible two-scale configurations relative to $w$, is the set of triplets $(u,E,P)$ with
$$
u\in BV(\Om'),\qquad E\in L^2(\Om'\times \ys;\R^N),\qquad P\in \ms_b(\Om'\times \ys;\R^N),
$$
such that 
\begin{equation}
\label{eq:wOm'bis}
u=w,\qquad E=\nabla w,\qquad P=0\qquad \text{ on }(\Om'\setminus\Omb)\times \ys,
\end{equation}
and also such that there exists $\mu\in \xs(\Om')$ (see \eqref{defX}) with
\begin{equation}
\label{eq:adm1}
E(x,y)\,(\leb^N_x\otimes\leb^N_y)+P-Du\otimes \leb^N_y= D_y\mu\qquad \text{in }\Om'\times \ys.
\end{equation}
\end{definition}

\begin{remark}
\label{rem:unique-mu}
The element $\mu\in \xs(\Om')$ associated with $(u,E,P)$ according to the previous definition is uniquely determined (see \cite[Remark 5.2 ]{FGhom}).
Moreover, following \cite[Lemma 4.8]{FGhom}, $\ashom(w,\Om')$ is closed under the associated weak convergences of the triplets $(u,E,P)$. \hfill\P
\end{remark}

The definition of the class of admissible two-scale configurations is motivated by the following compactness result (see \cite[Lemma 5.6]{FGhom}).

\begin{proposition}[\bf Compactness]
\label{lem:comp1}
Let  $(u_{\e},e_{\e},p_{\e})\in \as(w,\Om')$ be such that
$$
\|u_{\e}\|_{BV(\Om')}+\|e_{\e}\|_{L^2(\Om'; \R^N )}+\|p_{\e}\|_{\ms_b(\Om'; \R^N)}\le C.
$$
 Then, up to a subsequence,
\begin{equation}
\label{eq:def-weak2}
\begin{array}{ll}
u_{\e}\weak u&\qquad \text{weakly* in }BV(\Om')\\
e_{\e}\weakdue  E&\qquad \text{\duew in }L^2(\Om'\times \ys;\R^N)\\
p_{\e}\weakduest P&\qquad\text{\duewst in }\ms_b(\Om'\times \ys;\R^N)\\
\end{array}
\end{equation}
 with $(u,E,P)\in \ashom(w,\Om')$. 
\end{proposition}

 We will denote the convergences in \eqref{eq:def-weak2} as follows:
\begin{equation}
\label{eq:def-weak2-bis}
(u_{\e},e_{\e},p_{\e})\weakdue (u,E,P).
\end{equation}

For $(u,E,P)\in \ashom(w,\Om')$ we set
\begin{equation}
\label{eq:defQhom}
\qshom(E):=\frac{1}{2}\int_{\Om'\times\ys}\C(y)E\cdot E\,dxdy
\end{equation}
and
\begin{equation}
\label{eq:defHhom}
\hshom(P):=\int_{\Om'\times \ys}H\left( y,\frac{P}{|P|}\right)\,d|P|.
\end{equation}
We call $\qshom$ the  two-scale  homogenized elastic energy, and $\hshom$ the  two-scale  homogenized dissipation. 
Notice that the domain of integration  in the definition of $\hshom$ can be extended to $\Om'$ since $P=0$ on $(\Om'\setminus\Omb)\times\ys$.
\par
 The two-scale homogenized dissipation $\hshom$ is lower semicontinuous with respect to weak* convergence in $\ms_b(\Om'\times \ys;\R^N)$  as a consequence of  Reshetnyak's Lower Semicontinuity Theorem for functionals defined on measures (see \cite[Theorem 2.38]{ambrosio.fusco.pallara}).  This is so because the dissipation density $H$ in \eqref{eq:defHper} is lower semicontinuous on $(\Om'\times \ys)\times \R^N$ thanks to the ordering assumption \eqref{eq:ordK}.\EEE
\par
The following lower semi-continuity result  involving two-scale convergence  holds true (see \cite[Theorem 5.7]{FGhom}).

\begin{theorem}[\bf  Two-scale  lower semicontinuity]
\label{prop:lsc-hs}
Let $(u_\e,e_\e,p_\e)\in \as(w,\Om')$  be such that
\begin{equation*}
(u_{\e},e_{\e},p_{\e})\weakdue (u,E,P)
\end{equation*}
(see \eqref{eq:def-weak2-bis}) with $(u,E,P)\in \ashom(w,\Om')$. Then, for $\qs_\e$ and $\hs_\e$ as in \eqref{eq:defQe} and \eqref{eq:defHe} respectively, we get
\begin{equation*}
\qshom(E)\le  \liminf_{\e\to 0^+}\qs_\e(e_\e),
\end{equation*}
and
\begin{equation*}
\hshom(P)\le \liminf_{\e\to 0^+}\hs_\e(p_\e).
\end{equation*}
\end{theorem}

The results above allow us to introduce a kind of  two-scale $\Gamma$-limit which is the object of the next subsection.

\subsection{ Two-scale homogenization}\label{sub:2sc-gl}
\label{subsec:2s}
For $(u,E,P)\in \ashom(w,\Om')$ set
 \bel{eq:ehom}
 \eshom(u,E,P):=\qshom(E)+\hshom(P),
 \ee
  where $\qshom$ and $\hshom$ are defined in \eqref{eq:defQhom} and \eqref{eq:defHhom}, respectively.
We propose  to prove the 

\begin{theorem}\label{thm:2sc-gl}({\bf Two-scale homogenized energy as a ``two-scale $\Gamma$-limit''})
Assume that
 \bel{eq:reg-gd}
 \hn({\overline\Gamma}_d\setminus\Gamma_d)=0.
 \ee
 The following holds true:
\begin{itemize}
\item[(a)] If  a sequence $(u_\e,e_\e,p_\e)\in \as(w,\Om')$ is such that $(u_\e,e_\e,p_\e) \weakdue (u,E,P)$ (see \eqref{eq:def-weak2-bis}), then $(u,E,P)\in \ashom(w,\Om')$ and
$$
\eshom(u,E,P)\le \liminf_{\e\to 0^+} \es_\e(u_\e,e_\e,p_\e);
$$
\item[(b)] Given $(u,E,P)\in \ashom(w,\Om')$,  there exists a sequence $(u_\e,e_\e,p_\e)\in \as(w,\Om')$ with 
$(u_\e,e_\e,p_\e)\weakdue (u,E,P)$
such that
$$
\eshom(u,E,P)=\lim_{\e\to 0^+} \es_\e(u_\e,e_\e,p_\e).
$$
\end{itemize}
\end{theorem}

The proof of the first item in Theorem \ref{thm:2sc-gl} is immediate in view of Theorem \ref{prop:lsc-hs}. It remains to prove the second item in the theorem. This will be achieved through the concatenation of three lemmata.

\begin{lemma}\label{lem:trans1}
For any $(u,E,P)\in \ashom(0,\Om')$, there exists a sequence $(u_n,E_n,P_n)\in \ashom(0,\Om')$ with 
$$
\left\{\begin{array}{ll}
u_{n}\weakst u&\qquad \text{weakly* in }BV(\Om')\\[2mm]
E_{n}\rightarrow  E&\qquad \text{strongly in }L^2(\Om'\times \ys;\R^N)\\[2mm]
P_{n}\weakst P&\qquad\text{weakly* in }\ms_b(\Om'\times \ys;\R^N)
\end{array}\right.
$$
such that $|P_n|(\Om'\times \Sigma)=0$ and $\lim_{n\to+\infty }\hshom(P_n)=\hshom(P)$.
\end{lemma}

\begin{proof}
 We divide the proof into three steps.

\vskip10pt\noindent  Step 1-- {\sf Localization and translations on the torus.}
Consider $\mu\in \xs(\Om')$ associated with  $(u,E,P)$. 
\par
Recalling \eqref{eq:dec-mu} in Proposition \ref{lem:Xs}, we define the $y$-translation of $\mu$ by a vector $\tau$ as
$$
\mu_\tau=\mu_x(y-\tau)(\eta\otimes\leb^N_y).
$$
Then, testing $D_y\mu_\tau$ by $\ph(x)\Phi(y)$ with $\ph,\Phi$ smooth and compactly supported, it is easily seen that
$$
D_y\mu_\tau= E_\tau+P_\tau-Du\otimes \leb^N_y
$$
with $E_\tau$ the translate of $E(x,\cdot)$ in the direction $\tau$ and $P_\tau$ the push-forward of $P$ in the direction 
$0\times\tau$. We conclude that $\mu_\tau\in \xs(\Om')$ is associated with $(u, E_\tau, P_\tau)\in \ashom(0,\Om')$ in the sense of Definition \ref{set:Aw2}.

 Given $\psi\in C^\infty(\ys)$, 
$$
D_y(\psi\mu)= D_y\psi\mu+\psi(E+P-Du\otimes\leb^N_y)
$$
so that we deduce that $(u,\psi E, \psi P+(1-\psi)Du\otimes\leb^N_y+D_y\psi\mu)$ and  $(0,\psi E, \psi P-\psi Du\otimes\leb^N_y+D_y \psi\mu)$ are both elements of $\ashom(0,\Om')$, since $\mu=0$ on $(\Om'\setminus \Omb) \times \ys$. The associated element of $\xs(\Om')$  according to \eqref{eq:dec-mu}  is in both cases given by   $\psi\mu-\left(\int_\ys \psi(y)\mu_x(y)dy\right) \eta\otimes\leb^N_y$ , so as to enforce the zero mean condition in \eqref{eq:mux-bd}.

\EEE
\par
Fix $\e>0$  and consider a finite covering $\{V_i\}_{i=0,\dots,k}$ of $\ys$ by open sets in $\ys$ such that $\Sigma\subset\cup_{i=1,...,k} V_i$, $V_0\cap\Sigma=\emptyset$ and 
\begin{equation}
\label{eq:Peps}
|P|(\Om'\times (\cup_{i=1}^k V_i \setminus \Sigma))<\e.
\end{equation}
 Also consider an associated partition of unity $\{\psi_i\}_{i=0,..., k}$ with  $\psi_i(y)\in C^\infty_c(V_i),\; 0\le \psi_i\le 1$ and $\sum_{i=0,...,k}\psi_i=1$. Note that $\sum_{i=1}^k\psi_i=1$ in a neighborhood of $\Sigma$.  Because  the phase $\ys_1$ is Lipschitz, we may assume that  there exists for each $i\in \{1,...,k\}$ a translation $\tau_i$ such that,  for every  $0<\delta<1$, 
\bel{eq:trans-y1}
 ({\overline\ys}_1 \cap V_i) +\delta\tau_i\subset\!\subset \ys_1.
\ee
We now superimpose the two localizations of $\mu$ proposed above, that is, we localize through the $\psi_i$ adopting the first one  for $i=0$ and the second one for $i=1,\dots,k$, and then translate for a sequence $\delta_n\tau_i$  with $\delta_n\searrow 0$ \EEE for $i=1,\dots,k$. We obtain the sequence $(u,E_n,P_n)\in\ashom(0,\Om')$ with
\bel{eq:decomp}
\begin{cases}
E_n:= \psi_0E+\sum_{i=1}^k(\psi_iE)_{\delta_n\tau_i}\\[2mm]
P_n:=\psi_0P+(1-\psi_0)Du\otimes\leb^N_y+D_y\psi_0\mu+\sum_{i=1}^k(\psi_iP)_{\delta_n\tau_i}\\[2mm]\qquad\qquad\qquad\qquad-\sum_{i=1}^k(\psi_i)_{\delta_n\tau_i}Du\otimes\leb^N_y+\sum_{i=1}^k(D_y\psi_i\mu)_{\delta_n\tau_i}.
\end{cases}\ee

\vskip10pt\noindent Step 2 -- {\sf Convergence of the fields.}  Let $(u,E_n,P_n)\in\ashom(0,\Om')$ be the sequence constructed in the previous step.
It is clear that
$$
E_n\to E, \; \mbox{ strongly in } L^2(\Om'\times\ys;\R^N).
$$
Now, as far as $P_n$ is concerned, 
\begin{equation}
\label{eq:mu1n}
(1-\psi_0)Du\otimes\leb^N_y-\sum_{i=1}^k(\psi_i)_{\delta_n\tau_i}Du\otimes\leb^N_y\to 0 
\qquad\text{strongly in $\ms_b(\Om'\times \ys;\R^N)$}
\end{equation}
since $1-\sum_{i=0.,...,k}\psi_i=0$, and, in view of \eqref{eq:dec-mu},
\begin{equation}
\label{eq:mu2n}
D_y\psi_0\mu+\sum_{i=1}^k(D_y\psi_i\mu)_{\delta_n\tau_i}\to 0 
\qquad\text{strongly  in $\ms_b(\Om'\times \ys;\R^N)$}
\end{equation}
since translations converge strongly in $L^1(\Om'\times\ys; \eta\otimes \leb^N_y)$ and $\sum_{i=0}^kD_y\psi_i=0$. 
 Finally the term $\psi_0P+\sum_{i=1}^k(\psi_iP)_{\delta_n\tau_i}$
 goes to $P$ \wstar in $\ms_b(\Om'\times \ys;\R^N)$ for the same reason and since $\sum_{i=0}^k\psi_i=1.$
 We conclude that
$$
P_n\weakst P  \qquad\mbox{weakly* in } \ms_b(\Om'\times\ys;\R^N).
$$
Further, since, recalling \eqref{eq:dec-mu},  the measure 
$$(1-\psi_0)Du\otimes\leb^N_y+D_y\psi_0\mu-\sum_{i=1}^k(\psi_i)_{\delta_n\tau_i}Du\otimes\leb^N_y+\sum_{i=1}^k(D_y\psi_i\mu)_{\delta_n\tau_i}$$ 
is absolutely continuous with respect to $(|Du|+|\eta|)\otimes\leb^N_y$, that measure does not charge $\Om'\times  \Sigma$, so that, in view of \eqref{eq:decomp} and because $\psi_0\equiv 0$ on $\Sigma$,
\begin{multline*}
|P_n|(\Om'\times  \Sigma )\le \sum_{i=1}^k|(\psi_iP)_{\delta_n\tau_i}| (\Om'\times   \Sigma )=
\sum_{i=1}^k|(\psi_iP)| (\Om'\times ( \Sigma +\delta_n\tau_i))\\
\le \sum_{i=1}^k|P| (\Om'\times ( (\Sigma \cap V_i)+\delta_n\tau_i)).
\end{multline*}
$$
$$ 
We then choose $\delta_n\searrow 0$ such that $|P|(\Om'\times ( (\Sigma \cap V_i) +\delta_n\tau_i))=0$ for all $n$'s and $i$'s  (as the sets $(\Sigma \cap V_i)+\delta\tau_i)$ are mutually disjoint as $\delta$ varies, being translations of a Lipschitz graph) and obtain 
$$
|P_n|(\Om'\times\Sigma)=0.\medskip
$$

\vskip10pt\noindent Step 3 -- {\sf Convergence of the dissipation.}
We remark that, by sub-additivity of $\hshom$ and in view of \eqref{eq:decomp},
$$
\hshom(P_n)\le \hshom\bigg(\psi_0P+\sum_{i=1}^k(\psi_iP)_{\delta_n\tau_i}\bigg)+ r_n
$$
where
\begin{multline*}
r_n:= \hshom\bigg((1-\psi_0)Du\otimes\leb^N_y-\sum_{i=1}^k(\psi_i)_{\delta_n\tau_i}Du\otimes\leb^N_y\bigg)
+\hshom\bigg(D_y\psi_0\mu+\sum_{i=1}^k(D_y\psi_i\mu)_{\delta_n\tau_i}\bigg) \longrightarrow 0
\end{multline*} 
 because of \eqref{eq:mu1n} and \eqref{eq:mu2n},  together with the bound $\hshom(Q)\le c_4|Q|(\Om'\times \ys)$ (see \eqref{eq:bdsHper}).
We obtain
\begin{multline*}
\limsup_{n\to+\infty} \hshom(P_n)\le \limsup_{n\to+\infty}\hshom\bigg(\psi_0P+\sum_{i=1}^k(\psi_iP)_{\delta_n\tau_i}\bigg)\\
 \le \hshom\left(\psi_0P\right)+ \limsup_{n\to+\infty} \hshom\left( \sum_{i=1}^k (\psi_iP)_{\delta_n\tau_i} \right).
\end{multline*}

Now, because the translations $\delta_n\tau_i, i=1,...,k$ satisfy \eqref{eq:trans-y1} while the dissipation functional is given by $H_1$ on $\overline{\ys_1}$,  and since translation of measures yields strict convergence, Reshetnyak's continuity theorem (see e.g. \cite[Theorem 2.39]{ambrosio.fusco.pallara} or \cite{spector10}) implies that 
$$
\lim_{n\to+\infty}\hshom\left( \sum_{i=1}^k(\psi_iP \lfloor {(\Om'\times \Sigma)})_{\delta_n\tau_i}\right)=\hshom\left( \sum_{i=1}^k(\psi_iP \lfloor {(\Om'\times \Sigma)})\right).
$$
Recalling \eqref{eq:Peps} we may thus write, for some $C>0$ independent of $n$, 
\begin{multline*}
\limsup_{n\to+\infty}\hshom\left( \sum_{i=1}^k(\psi_iP)_{\delta_n\tau_i}\right)\le \limsup_{n\to+\infty}\hshom\left(\sum_{i=1}^k (\psi_iP \lfloor {(\Om'\times \Sigma)})_{\delta_n\tau_i}\right)+C\e\\
=\hshom\left( \sum_{i=1}^k(\psi_iP \lfloor {(\Om'\times \Sigma)})\right)+C\e.
\end{multline*}
Consequently,
$$
\limsup_{n\to+\infty} \hshom(P_n)\le \hshom\left(\psi_0P\right)+\hshom\left( \sum_{i=1}^k(\psi_iP \lfloor {(\Om'\times \Sigma)})\right)+C\e.
$$
Now $\psi_0 P$ and $ \sum_{i=1}^k(\psi_iP \lfloor {(\Om'\times \Sigma)}$ have disjoint supports so that we infer that
$$
\limsup_{n\to+\infty} \hshom(P_n) \le \hshom\left(\psi_0P+ \sum_{i=1}^k(\psi_iP \lfloor {(\Om'\times \Sigma)})\right)+C\e.
$$
Using \eqref{eq:Peps} again we finally obtain
$$
\limsup_{n\to+\infty} \hshom(P_n)  \le \hshom \left(\sum_{i=0}^k \psi_i P\right)+2C\e.
$$
Letting $\e\to 0$, and using a diagonal argument, we conclude that the approximating sequence $(u,E_n,P_n)$ satisfies
$$
\limsup_{n\to+\infty} \hshom(P_n) \le \hshom(P).
$$
In view of the lower semicontinuity of $\hshom$ with respect to weak* convergence we get
$$
\lim_{n\to+\infty}\hshom(P_n)=\hshom(P),
$$
so that the proof is concluded.
\EEE
 \end{proof}
 
 \bigskip
 \begin{lemma}\label{lem:trans2}
 Assume  \eqref{eq:reg-gd}.
 Consider $(u,E,P)\in \ashom(0,\Om')$ with $|P|(\Om'\times\Sigma)=0$. There exists a sequence $(u_n,E_n,P_n)\in \ashom(0,\Om')$ with 
$$
\left\{\begin{array}{ll}
u_{n}\weakst u&\qquad \text{weakly* in } BV(\Om')\\[2mm]
E_{n}\to  E&\qquad \text{strongly in }L^2(\Om'\times \ys;\R^N)\\[2mm]
P_{n}\weakst P&\qquad\text{ weakly*  in } \ms_b(\Om'\times \ys;\R^N)
\end{array}\right.
$$
such that $(u_n, E_n,P_n)\equiv 0$  on a neighborhood $U_n$ of  $\overline{\Gamma}_d$, $|P_n|(\Om'\times\Sigma)=0$  and $\hshom(P_n)\rightarrow \hshom(P).$
 \end{lemma}

 \begin{proof}
 Let $\e>0$ and   introduce a finite covering $\{V_i\}_{i=1,...,k}$ of ${\overline \Gamma}_d$ by open sets  in such a way that
$$
\int_{(\cup_{i=1}^k V_i) \cap (\partial \Om \setminus \Gamma_d)}|u|\,d\hn=
\int_{(\cup_{i=1}^k V_i) \cap (\partial \Om \setminus \bar \Gamma_d)}|u|\,d\hn
<\e,
$$
where \eqref{eq:reg-gd} has been used in the equality above. 

Consider $\{\psi_i\}_{i=1,..., k}$
with  $\psi_i\in C^\infty_c(V_i),\; 0\le \psi_i\le 1$ and $\sum_{i=1}^k\psi_i=1$ in a neighborhood of ${\overline \Gamma}_d$.  Because $\pa\Om$ is Lipschitz,  we can assume that the $V_i$ have been chosen so that  there exists for each $i\in \{1,...,k\}$ a  translation vector $\tau_i \in \R^N$ such that, for  $0<\delta<1$,
\bel{eq:trans-x1}
({\overline\Om} \cap V_i)+\delta\tau_i\subset\!\subset \Om.
\ee

According to \eqref{eq:wOm'bis} and \eqref{eq:adm1},  $(u,E,P)$ and the associated $\mu$ are identically $0$ on $(\Om'\setminus{\overline\Om}) \times \ys$. We     extend $(u,E,P)$ and $\mu$  to $\R^N\times \ys$  as follows.    We set $u,E,P,\mu$  to be identically $0$ on $(\R^N\setminus{\overline\Om})\times \ys$  and define $P,\mu$ on $(\pa\Om\setminus\Gamma_d)\times \ys$ to be respectively
\begin{equation}
\label{eq:Pbdry}
P:=-u \nu \hn\lfloor (\pa\Om\setminus\Gamma_d) \otimes \leb^N_y \qquad\text{and}\qquad \mu:=0
\end{equation}
Such an extension  preserves the compatibility relation \eqref{eq:adm1}
 on $\R^N\times \ys$. Further, $|P|(\R^N\times\Sigma)=0$.
\par
Correspondingly, for each $i=1,\dots,k$,  the triplet 
$
\left(\psi_i u, \psi_iE,\psi_iP+D_x\psi_iu
\otimes\leb^N_y\right)
$
 satisfies the same compatibility condition with respect to $\psi_i \mu$.
\par
We translate that triplet by a sequence $\delta_n\tau_i$ with $\delta_n\searrow 0$  while leaving the remaining contribution $\left(1-\sum_{i=1}^k\psi_i\right)(u,E,P)$ fixed. Summing all contributions  generates a sequence $(u_n,E_n,P_n)$ with
\bel{eq:transx3}
\left\{\begin{array}{lccl}
u_n&\weakst& u&\text{weakly* in }BV(\R^N)\\[2mm]
E_n&\to& E&\text{ strongly  in }L^2(\R^N\times \ys;\R^N)\\[2mm]
P_n&\weakst& P&\text{weakly* in }\ms_b(\R^N\times \ys;\R^N)
\end{array}\right.
\ee
as $n\to+\infty$. 
\par
 Restricting $(u_n,E_n,P_n)$ to $\Om'\times\ys$ yields an element of $\ashom(0,\Om')$ identically vanishing in a neighborhood $U_n$ of $\Gamma_d$. Moreover, $|P_n|(\Om'\times \Sigma)=0$.
\par
Further, note that, for  $\eta\in \ms_b(\R^N\times\ys)$ and $\tau\in\R^N$, the $x$-translate $\eta_{\delta_n\tau}$ of $\eta$ by $\delta_n\tau$ satisfies for every open set $A\subseteq \R^N$
$$
\limsup_{n\to+\infty}|\eta_{\delta_n\tau}|(A\times\ys)\le |\eta|(\overline A\times\ys).
$$
Thus, because $P_n$ only involves translations and since the compatibility condition is satisfied by each $\psi_iP$,  \eqref{eq:Pbdry} implies that 
\begin{multline}\label{eq:transx6}
\limsup_{n\to+\infty}|P_n|(\Om'\times\ys)\le |P|(\Om'\times\ys)+\int_{\pa\Om\setminus\Gamma_d}
\left|\left(\sum_{i=1}^k\psi_i\right)u\right|d\hn\\[2mm]
 \le |P|(\Om'\times\ys)+\int_{(\cup_{i=1}^k V_i)\cap (\partial\Om \setminus \Gamma_d)}|u|\,d\hn< |P|(\Om'\times\ys)+\e.
\end{multline}
 Letting $\e\to 0$ and using a diagonal argument, we can assume that $\limsup_{n\to+\infty}|P_n|(\Om'\times \ys)\le |P|(\Om'\times \ys)$,   so that
\bel{eq:scP_n}
|P_n|(\Om'\times \ys)\to |P|(\Om'\times \ys).
\ee
Since $|P|$ does not charge $\Om'\times \Sigma$ and in view of the lower semicontinuity of the total variations on open sets, \eqref{eq:scP_n} yields 
\begin{equation}
\label{eq:mass1}
\begin{cases}|P_n|(\Om'\times \ys_1) \to |P|(\Om'\times \ys_1)
\\[3mm]
|P_n| (\Om'\times \ys_2) \to |P|(\Om'\times \ys_2).\end{cases}
\end{equation}
By construction $P_n$ does not charge $\Om'\times \Sigma$ and, further, the dissipation on the phases is given by the continuous functions $H_1$ and $H_2$.  In view of \eqref{eq:transx3}, \eqref{eq:mass1},  Reshetnyak's continuity theorem implies 
\begin{multline*}
\lim_{n\to+\infty} \hshom(P_n)=\lim_{n\to+\infty} \hshom(P_n \lfloor (\Om'\times \ys_1))+\lim_{n\to+\infty} \hshom(P_n \lfloor (\Om'\times \ys_2))\\
=\hshom(P \lfloor (\Om'\times \ys_1))+\lim_{n\to+\infty} \hshom(P\lfloor (\Om'\times \ys_2))=\hshom(P).
\end{multline*}
This concludes the proof.
\end{proof}

\begin{remark}
Inequality \eqref{eq:transx6} is the only place where the regularity \eqref{eq:reg-gd} is used in this work.
So lemmata \ref{lem:trans1} above and \ref{lem:trans3} below hold true absent \eqref{eq:reg-gd}.\hfill\P\end{remark}

In the following Lemma we  employ the regularization by convolution of measures on the torus $\ys$, which, as shown in the Appendix,  is  analogous to that of the euclidean setting. By a smooth configuration in $\ashom(0,\Om')$, we mean that the displacement and the elastic strain are smooth functions, and that the plastic strain is a measure absolutely continuous with respect to $\leb^N_x\otimes\leb^N_y$ with a smooth density.

 \begin{lemma}\label{lem:trans3}
 Consider $(u,E,P)\in \ashom(0,\Om')$ with $|P|(\Om'\times\Sigma)=0$ and $(u,E,P)\equiv 0$ in a neighborhood of 
 $\overline{\Gamma}_d$. There exists a smooth sequence $(u_n,E_n,P_n)\in \ashom(0,\Om')$ with
$$
\left\{\begin{array}{ll}
u_{n}\weakst u&\qquad \text{weakly* in } BV(\Om')\\[2mm]
E_{n}\to E&\qquad \text{strongly in }L^2(\Om'\times \ys;\R^N)\\[2mm]
P_{n}\weakst P&\qquad\text{weakly* in } \ms_b(\Om'\times \ys;\R^N)
\end{array}\right.
$$
such that $(u_n, E_n,P_n)\equiv 0$  on a neighborhood $U_n$ of $\overline{\Gamma}_d$ and $\hshom(P_n)\to \hshom(P)$.
 \end{lemma}
\begin{proof}
 We proceed by localization in the variable $x$ together with regularization by convolutions in both $x$ and $y$ mimicking the standard density result for smooth functions in $BV$ (see \cite[Theorem 3.9]{ambrosio.fusco.pallara}).
\par
Define for $i\ge 1$ 
$$
\Om'_i:=\{x\in\Om': {\rm dist }(x,\pa\Om')>1/i\},
$$
and set
$$
A_1=\Om'_2, \qquad A_i:=\Om'_{i+1}\setminus{\overline\Om'}_{i-1} \text{ for }i\ge 2.
$$

Consider $\psi_i\in C^\infty_c(A_i)$ with $0\le\psi_i\le 1$ and $\sum_i\psi_i=1$ on $\Om'$ and note that, if $\mu \in \xs(\Om')$ is associated to $(u,E,P)$, then 
$
(\psi_iu,\psi_iE,\psi_iP+D_x\psi_iu\,\leb^N_x\otimes\leb^N_y)\in\ashom(0,\Om')
$
with associated $\psi_i \mu\in\xs(\Om')$. 
\par
We regularize by convolutions  each triplet above  in both $x$ and $y$ using regularization kernels $\rho^x_{i,n}(x)$ and $\rho^y_n(y)$ and superpose the result. 
Setting
 $$
 \begin{cases}
 u_n:=\sum_i(\psi_iu)\star\rho^x_{i,n}\\[2mm]
 E_n:=\sum_i(\psi_iE)\star(\rho^x_{i,n}\rho^y_n)\\[2mm]
 P_n:=\sum_i\{(\psi_iP)\star(\rho^x_{i,n}\rho^y_n)+(D_x\psi_iu\star\rho^x_{i,n})\} \leb^N_x\otimes\leb^N_y
 \end{cases}
 $$
 together with the associated $\mu_n:=\sum_i\{(\psi_i\mu)\star(\rho^x_{i,n}\rho^y_n)\}\leb^N_x\otimes\leb^N_y$, 
  we can tune the kernels $\rho^x_{i,n}(x)$ so that $(u_n, E_n, P_n)$ belongs to $\ashom(0,\Om')$, vanishes in a neighborhood $U_n$ of $\overline{\Gamma}_d$ -- since $(u,E,P)$ vanishes in a neighborhood $U$ of $\overline{\Gamma}_d$ -- and satisfies
 $$
\left\{\begin{array}{ll}
u_{n}\weakst u&\qquad \text{weakly* in }BV(\Om')\\[2mm]
E_{n}\to E&\qquad \text{strongly in }L^2(\Om'\times \ys;\R^N)\\[2mm]
P_{n}\weakst P&\qquad\text{weakly* in }\ms_b(\Om'\times \ys;\R^N)
\end{array}\right.
$$
with
$$
\left|\sum_i(D_x\psi_iu\star\rho^x_{i,n})\,\leb^N_x\otimes  \leb^N_y \right|(\Om'\times\ys)\to 0
$$
since $\sum_iD_x\psi_i=0$. 
 Moreover, in view of e.g. \cite[Theorem 2.2(b)]{ambrosio.fusco.pallara}, we can also require
$$
\left|[(\psi_i P)\star(\rho^x_{i,n}\rho^y_n)]\leb^N_x\otimes \leb^N_y\right|(\Om'\times\ys)\le \left|\psi_i P\right|(\Om'\times \ys),
$$
so that
\begin{multline*}
 \limsup_{n\to+\infty}|P_n|(\Om'\times\ys) 
 \le \sum_i  \limsup_{n\to+\infty}\left|[(\psi_i P)\star(\rho^x_{i,n}\rho^y_n)]\leb^N_x\otimes \leb^N_y\right|(\Om'\times\ys)\\
 +\limsup_{n\to+\infty}\left|\sum_i(D_x\psi_iu\star\rho^x_{i,n})\,\leb^N_x\otimes\leb^N_y\right|(\Om'\times\ys)
 \le \sum_i \left|\psi_i P\right|(\Om'\times \ys)=|P|(\Om'\times\ys).
\end{multline*}
We conclude that 
$$
|P_n|(\Om'\times \ys)\to |P|(\Om'\times \ys).
$$
Since $P_n$ and $P$ do not charge $\Om'\times \Sigma$,  the proof that  $\hshom(P_n)\rightarrow \hshom(P)$ is identical to that at the end of Lemma \ref{lem:trans2}.
\end{proof}
 
We are now in a position to prove Theorem \ref{thm:2sc-gl}.

\begin{proof}
Consider a triplet $(u,E,P)\in \ashom(0,\Om')$. Using successively the three previous lemmata and a diagonalization 
process, we find a neighborhood  $U_n$ of $\overline{\Gamma}_d$  and a {\it smooth} sequence $(u_n,E_n,P_n)\in \ashom(0,\Om')$  (and an associated smooth $\mu_n\in \xs(\Om')$) with $(u_n,E_n,P_n)\equiv 0$ on $U_n\times\ys$, $|P_n|( \Om'\times \Sigma)=0$,
$$
\left\{\begin{array}{ll}
u_{n}\weakst u&\qquad \text{weakly* in }BV(\Om')\\[2mm]
E_{n}\to  E&\qquad \text{strongly in }L^2(\Om'\times \ys;\R^N)\\[2mm]
P_{n}\weakst P&\qquad\text{weakly* in }\ms_b(\Om'\times \ys;\R^N)
\end{array}\right.
$$
 and such that
$$
\hshom(P_n) \to \hshom(P).
$$
Note that the associated  $\mu_n\in\xs(\Om')$ also vanishes identically on $U_n\times\ys$.
\par
Take a sequence $\ph_m\in C^\infty_c(\Om')$ with $\ph_m\nearrow \chi_{\Om'}$ and consider
$$
(u_n,\ph_mE_n,\ph_m P_n+(1-\ph_m)Du_n\otimes\leb^N_y)\in\ashom(0,\Om')
$$
 with associated $\varphi_m  \mu_n \in \xs(\Om')$.
Clearly, as $m\to+\infty$,
$$\left\{\begin{array}{ll}
\ph_mE_{n}\to  E_n&\qquad \text{strongly in }L^2(\Om'\times \ys;\R^N)\\[2mm]
\ph_m P_n+(1-\ph_m)Du_n\otimes\leb^N_y\to P_n&\qquad\text{strongly in }L^1(\Om'\times \ys;\R^N)
\end{array}\right.
$$
so, through yet another diagonalization process, we can assume  in addition that,  near $\pa\Om'\times\ys$, $E_n,\mu_n\equiv 0$ and that $P_n$ is   independent of $y$.
\par
 We now construct the recovery sequence for item (b) in the statement of the theorem.
Set
$$
u_{n,\e}(x)=u_n(x)-\e\mu_n\left(x,\xoe\right),
$$
so that
\begin{multline}
\label{eq:eq-kce}
Du_{n,\e}(x)=Du_n(x)-\e D_x\mu_n\left(x,\xoe\right)- D_y\mu_n\left(x,\xoe\right)\\
=E_n\left(x,\xoe\right)+P_n\left(x,\xoe\right)- \e D_x\mu_n\left(x,\xoe\right).
\end{multline}
We can define
$$
\begin{cases}
e_{n,\e}:=E_n(x,\xoe)-\e D_x\mu_n(x,\xoe)\\[1mm]
p_{n,\e}:=P_n(x,\xoe)
\end{cases}
$$
and  have thus constructed a triplet $(u_{n,\e}, e_{n,\e}, p_{n,\e})$. Further, in view of \eqref{eq:eq-kce} and since $\mu_n\equiv0$ near $\overline{\Gamma}_d$, that triplet belongs to $\as(0,\Om')$.  
\par
Since $e_{n\,e}\equiv 0$ and $p_{n,\e}$ does not depend on $x/\e$ near $\partial \Om'$,   a Riemann-Lebesgue type Lemma (see e.g. \cite[Lemma 5.5 and Lemma 5.6]{allaire}) immediately implies that,
$$
\left\{\begin{array}{rcll}
u_{n,\e}&\weakst& u_n,&\text{weakly* in } BV(\Om')\\[2mm]
e_{n,\e}&\weakdue& E_n, &\text{ two scale  weakly in }L^2(\Om'\times\ys;\R^N)\\[2mm]
p_{n,\e}&\weakduest& P_n, &\text{ two scale   weakly* in }\ms_b(\Om'\times\ys;\R^N)\\[2mm]
\end{array}\right.
$$
as $\e\searrow 0$ with
$$
\qs_\e(e_{n,\e})\to\qshom(E_n)\qquad\text{and}\qquad
\hs_\e(p_{n,\e})\to\hshom(P_n).
$$
A standard diagonalization process produces a triplet $(u_{\e},e_{\e},p_{\e})\in\as(0,\Om')$  which is the desired recovery sequence for $(u,E,P) \in \ashom(0,\Om')$.
\par
If $w\in H^1(\R^N)\ne 0$,  it is enough  to approximate $(u-w,E-\nabla w,P)\in \ashom(0,\Om')$ and to translate the approximating  triplet by $(w,\nabla w,0)$. The proof is complete.
\end{proof}

\subsection{ The homogenization result}
\label{sub:homHencky}

In this section we characterize  the asymptotic behaviour of  the elasto-plastic functionals $\es_\e$ (see \eqref{eq:totEeps}) in terms of  the two-scale homogenized energy $\eshom$ given in \eqref{eq:ehom}.
\par
For every $(u,e,p)\in \as(w,\Om')$  consider
\begin{multline}
\label{eq:Ehom}
\es^{hom}(u,e,p):=\min\{\eshom(u,E,P):  (u,E,P)\in  \ashom(w,\Om') \\
\text{ and \eqref{eq:Ee} and \eqref{eq:Pp} are satisfied}\},
\end{multline}
where
\begin{equation}
\label{eq:Ee}
e(x)=\int_\ys E(x,y)\,dy\qquad\text{ for $\leb^N$-a.e. $x\in \Om'$},
\end{equation}
and
\begin{equation}
\label{eq:Pp}
p(B)=P(B\times \ys)\qquad \text{for every Borel set $B\subseteq \Om'$}.
\end{equation}
The minimum problem in \eqref{eq:Ehom} is well posed.  Indeed, $\eshom$ is coercive and lower semicontinuous in $(E,P)$ and $\ashom(w,\Om')$ is closed under weak/weak* convergences of the strains (see Section \ref{sub:Henky2s}), while relations \eqref{eq:Ee} and \eqref{eq:Pp} are stable, taking $\varphi(x)\in C^0_c(\Om')$  as a test for weak/weak* convergence.
\par
 We extend $\es_\e$ and $\es^{hom}$ to $BV(\Om')\times L^2(\Om';\R^N)\times \ms_b(\Om';\R^N)$ by setting $\es_\e=+\infty$ and $\es^{hom}=+\infty$ outside $\as(w,\Om')$ and note that  the $\es_\e$ are equicoercive since \eqref{eq:Kin} holds true for $\es_\e$  for some $C>0$ independent of $\e$. 
The following result holds.

\begin{theorem}[\bf The Homogenization Result]
\label{thm:hom-res}
Assume that \eqref{eq:reg-gd} is satisfied. Then $(\es_\e)$ $\Gamma$-converges to $\es^{hom}$ given in \eqref{eq:Ehom} $\e\to 0^+$ with respect to the product of the weak*-$BV(\Om')$, weak-$L^2(\Om';\R^N)$ and weak*-$\ms_b(\Om';\R^N)$ topologies.
\end{theorem}

\begin{proof}
The proof is a simple consequence of the results of the Subsection \ref{subsec:2s}. 
\par
Assume $(u_\e,e_\e,p_\e) \weak (u,e,p)$. By Proposition \ref{lem:comp1}, and recalling notation \eqref{eq:def-weak2-bis},  there exists $(u,E,P)\in \ashom(w,\Om')$ such that, up to a subsequence,
$$
(u_\e,e_\e,p_\e)\weakdue (u,E,P).
$$
In particular relations \eqref{eq:Ee} and \eqref{eq:Pp} are satisfied thanks to the general properties of two-scale convergence (see Remark \ref{rem:marginal}). Moreover, in view of the lower semicontinuity in Theorem \ref{prop:lsc-hs} and of the very definition of $\es^{hom}$ we may write
$$
\liminf_{\e\to 0^+} \es_\e(u_\e, e_\e, p_\e) \ge \eshom(u,E,P)\ge \es^{hom}(u,e,p).
$$
\medskip

 Given $(u,e,p)\in \as(w,\Om')$, let $(u,E,P)\in \ashom(w,\Om')$ satisfying \eqref{eq:Ee} and \eqref{eq:Pp} and such that
$$
\es^{hom}(u,e,p)=\eshom(u,E,P)
$$
according to \eqref{eq:Ehom}. Thanks to item (b) in Theorem \ref{thm:2sc-gl}, there exists a sequence $(u_{\e},e_{\e},p_{\e})\in \as(w,\Om')$ with $(u_{\e},e_{\e},p_{\e}) \weakdue (u,E,P)$ and such that
$$
\eshom(u,E,P)=\lim_{\e\to 0^+} \es_{\e}(u_{\e},e_{\e},p_{\e}),
$$
which is the desired recovery sequence.
\end{proof}

\section{Is the homogenized energy elasto-plastic?}\label{sec:notplastic}

In this last section, we investigate  to which extent the homogenized functional $\es^{hom}(u,e,p)$ defined in \eqref{eq:Ehom} is still an elasto-plastic energy  of the form \eqref{eq:estot} involving an elastic energy associated to the elastic strain $e$ and a dissipation potential associated to the plastic strain $p$. This is far from  obvious since the minimum problem in the two-scale setting defining $\es^{hom}$ is nonlinear.
\par
To tackle this issue, it proves convenient to resort to a  one-field description of $\es_\e$ and of $\es^{hom}$; see \cite{dalmaso.toader} for an attempt at  a similar description for evolution in a model combining plasticity and fracture.  

\subsection{The one-field viewpoint for heterogeneous plasticity}\label{sub:of-het}
 \label{subsec:one-filed}
Returning to the Hencky two-phase setting of Section \ref{sec:Hencky-het} we
set, for every $u\in BV(\Om')$ with $u=w$ on $\Om'\setminus \Omb$,
$$
\fs(u):=\min_{(e,p)}\{\es(u,e,p)\,:\, (u,e,p)\in \as(w,\Om')\}\\
=\min_{(e,p)}\{\qs(e)+\hs(p)\,:\, (u,e,p)\in \as(w,\Om')\}.
$$
The minimum problem is well posed and admits a unique solution, since the elastic energy is quadratic in $e$ (hence strictly convex) and in view of the compatibility condition $Du=e+p$.
\par
We can provide an integral representation for $\fs$ in the following way.  For any $\xi\in \R^N$ and $i\in \{1,2\}$  set
\bel{eq:eq-infs}
W_i(\xi):=\min_{p\in\R^N} \left\{\frac{1}{2}\C_i (\xi-p)\cdot (\xi-p)+ H_i(p)\right\}= \max_{\sigma\in K_i}\left\{\xi\cdot\sigma-\frac{1}{2}\C_i^{-1}\sigma\cdot\sigma\right\}
\ee
where the second equality can easily be checked through classical arguments of convex analysis. 
Consider the Caratheodory function $W:(\Om  \cup \Gamma_d)\times\R^N\to \R$ defined as
\bel{eq:defWx}
W(x,\xi):=
\begin{cases}
W_1(\xi) &\text{ if $x\in\Om_1\cup \Sigma \cup (\pa\Om_1\cap \Gamma_d)$}\\[2mm]
W_2(\xi) &\text{ otherwise }.
\end{cases}
\ee
The recession function $W^\infty(x,\xi)$ of $W(x,\xi)$, that is
$$
W^\infty(x,\xi):= \limsup_{t\to+\infty}\frac{W(x,t\xi)}t,
$$
 satisfies
\bel{eq:H=Winf}
W^\infty(x,\xi)=H(x,\xi).
\ee
 In the spirit of \cite[Lemma 4.1]{dalmaso.toader}, if $u\in BV(\Om')$, 
$$
W(x,\nabla u(x))=\frac{1}{2}\C(x) (\nabla u(x)-\overline p(x)\cdot (\nabla u(x)-{\overline p}(x))+ H(x,{\overline p}(x)), \mbox{ a.e. in }\Om,
$$
where $\overline p$ is measurable since the minimizer $p$ in \eqref{eq:eq-infs} is unique and depends continuously on $\xi$ and integrable   in view of \eqref{eq:eq-infs}. Setting $\bar p=0$ on $\Om'\setminus \Omb$, defining 
$$
e_u:=\nabla u-{\overline p}\qquad\text{and}\qquad p_u={\overline p}\leb^N+Du^s
$$ 
and recalling \eqref{eq:H=Winf} we recover a triplet $(u,e_u,p_u)\in \as(w,\Om')$ such that 
\begin{multline*}
\fs(u)=\es(u,e_u,p_u)=\qs(e_u)+\hs(p_u)
\\=\int_\Om W(x,\nabla u)\,dx+\int_{\Om}  H \left( x,\frac{D^su}{|D^su|}\right)\,d|D^su|\\
+\int_{\Gamma_d} H(x, (w-u)\nu)\,d\hn+\frac12\int_{\Om'\setminus\Om}\C \nabla w\cdot \nabla w dx.
\end{multline*}

The equilibrium problem \eqref{eq:minpb} can thus be rephrased as a minimum problem in the displacement $u$ as
$$
\min\{\fs(u)\,:\, u\in BV(\Om'), u=w \text{ on }\Om'\setminus \Omb\}.
$$

We will make use of the following lemma.

\begin{lemma}
\label{rem:fs-rel}
The functional $\fs(u)$ is the lower semi-continuous envelope  with respect to  the strong $L^1$-topology of
$$
 u\mapsto 
 \begin{cases}
 \displaystyle\int_\Om W(x, \nabla u)\; dx+\frac12\int_{\Om'\setminus\Om}\C \nabla w\cdot \nabla w dx,&\text{if $u\in W^{1,1}(\Om')$ with $u=w$ on $\Om'\setminus \Omb$,}\\[2mm]
 +\infty,&{\rm otherwise.}
 \end{cases}
$$
\end{lemma}

\begin{proof}
The proof follows by application of the relaxation result \cite[Theorem 1.1]{ADF}  upon taking into account the boundary value $w$. 

 Let $k\in\N$, and, for any Borel set $B\subset\R^N$, let $W_{k,B}: B\times \R^N\to \R$ be given by
$$
W_{k,B}(x,\xi):=
\begin{cases}
W_1(\xi)& \text{if }x \in B\cap\Omb_1\\
W_2(\xi)& \text{if }x \in B\cap\Omb\setminus \Omb_1\\
k|\xi| &\text{if }x\in B\setminus \Omb.
\end{cases}
$$
Consider an open set $A$  such that $\Om'\subset\subset A$. 
Note that $W_{k,A}=W$ on $(\Om\cup\Gamma_d)\times \R^N$, and that $W_{k,A}$ satisfies relation (1.6)  in \cite[Theorem 1.1]{ADF} for $k$ larger than $c_2$ defined in \eqref{eq:Kinc} (since $W_1\le W_2$). By \cite{ADF} we conclude that 
$$
\fs_{k,A}(v):=
\begin{cases}
\int_A W_{k,A}(x,\nabla v)\,dx+\int_{A}  W_{k,A}^{\infty} \left( x,\frac{D^sv}{|D^sv|}\right)\,d|D^sv|&u\in BV(A)\\
+\infty &v\in L^1(A)\setminus BV(A)
\end{cases}
$$
is the relaxed functional with respect to the $L^1$-topology of
\begin{equation*}
v\mapsto 
\begin{cases}
\int_A W_{k,A}(x,\nabla v)\,dx &v\in W^{1,1}(A)\\
+\infty&v\in L^1(A)\setminus W^{1,1}(A).
\end{cases}
\end{equation*}

With this setup the liminf inequality for $\fs$ is obvious since  it suffices to localize the liminf inequality for  $\fs_{k,A}$ to $\Om'$ because the contribution on $\Om'\setminus \Om$ is fixed and given by $k\int_{\Om'\setminus \Om}|\nabla w|\,dx$.

The limsup inequality for $\fs$ goes as follows.  Let $u\in BV(\Om')$ with $u=w$ on $\Om'\setminus \Omb$ and, thanks to  Gagliardo's trace theorem, let $\tilde u\in BV(A)$ be an extension of $u$ to $A$, which belongs to $W^{1,1}$ on $A\setminus \Omb$ and whose trace (from outside) on $\partial \Om$ is given by $\tilde w:=w1_{\Gamma_d}+u 1_{\partial\Om\setminus \Gamma_d}$.  By the relaxation result, we can find $\tilde u_{k,n}\in W^{1,1}(A)$ with $\tilde u_{k,n}\stackrel{n}{\to} \tilde u$ strongly in $L^1(A)$ such that
$$
\lim_{n\to+\infty} \int_A W_{k,A}(x,\nabla \tilde u_{k,n})\,dx=\fs_{k,A}(\tilde u).
$$
Note that given $B\subseteq A$ open, the optimality of $\tilde u_{k,n}$ on $A$ implies (thanks to the liminf inequality on $A\setminus \bar B$)
$$
\limsup_{n\to+\infty} \int_B W_{k,A}(x,\nabla \tilde u_{k,n})\,dx \le \fs_{k,A\cap\bar B}(\tilde u),
$$
\par
Using this we get localizing on $A\setminus \Omb$ and dividing by $k$
$$
\limsup_{n\to+\infty} \int_{A\setminus \Omb} |\nabla \tilde u_{k,n}|\,dx\le \int_{A\setminus \Omb}|\nabla \tilde u|\,dx+\frac{1}{k}\int_{\Gamma_d} H(x, (w-u)\nu)\,d\hn,
$$
while, localizing on $\Om$,
\begin{multline*}
\limsup_{n\to+\infty}
\int_{\Om} W(x,\nabla \tilde u_{k,n})\,dx\\
\le \int_{\Om} W(x,\nabla u)\,dx+\int_{\Om} H \left( x,\frac{D^s u}{|D^s u|}\right)\, d|D^su|+\int_{\Gamma_d}H(x,(w-u)\nu)\,d\hn.
\end{multline*}
In both inequalities we used that by construction $D^s\tilde u\lfloor \partial\Om=(w-u)\nu\hn\lfloor \Gamma_d$ and $W_{k,A}^\infty=H$ on $\Om\cup\Gamma_d$.
A diagonal argument yields the existence of a sequence $n_k\to +\infty$ such that $\tilde u_{k,n_k}\to \tilde u$ strongly in $L^1(A)$, 
$ \int_{A\setminus\Omb}|\nabla \tilde u_{k,n_k}|dx\to\int_{A\setminus\Omb}|\nabla \tilde u|dx$ and
\begin{multline*}
\limsup_{k\to+\infty} \int_{\Om} W(x,\nabla \tilde u_{k,n_k})\,dx
\\ 
\le \int_{\Om} W(x,\nabla u)\,dx+\int_{\Om} H \left( x,\frac{D^s u}{|D^s u|}\right)\, d|D^su|+\int_{\Gamma_d}H(x,(w-u)\nu)\,d\hn.
\end{multline*}
In particular $\tilde u_{k,n_k}\to \tilde w$ strongly in $L^1(\partial\Om)$. Let us consider $G_k\in W^{1,1}(\Om)$ such that $G_k\to 0$ strongly in $W^{1,1}(\Om)$ and its trace on $\partial\Om$ is given by $\tilde w-\tilde u_{k,n_k}$. 
\par
The recovering sequence for $u$   relative to $\fs$ that   we are seeking is then given by $u_k\in W^{1,1}(\Om')$ such that $u_k=\tilde u_{k,n_k}+\tilde G_k$ on $\Om$ and $u_k=w$ on $\Om'\setminus \Omb$.
\end{proof}

\subsection{Periodic heterogeneous plasticity: the  homogenized one-field functional}
\label{sub:of-het-per}
Consider now the periodic setting of Subsection \ref{sub:unit-cell}. Introduce  $W:\ys\times \R^N\to [0,+\infty[$ given by
$$
W(y,\xi):=
\begin{cases}
W_i(\xi)&\text{if }y\in\ys_i\\[2mm]
W_1(\xi)&\text{if }y\in  \Sigma,
\end{cases}
$$
with $W_i$ defined in \eqref{eq:eq-infs}, and the associated
$\displaystyle
W_\e(x,\xi):= W\left(\frac x\e,\xi\right).
$
Recalling \eqref{eq:defQe} and  \eqref{eq:defHe} the one-field functional 
\begin{equation}
\label{eq:oneFe}
\fs_\e(u)=\min_{(e,p)} \left\{\es_\e(u,e,p)\,:\, (u,e,p)\in \as(w,\Om')\right\}
\end{equation}
takes the form
\begin{multline*}
\fs_\e(u)=\int_\Om W_\e(x,\nabla u)\,dx+ \hs_\e(D^su)+ \frac{1}{2}\int_{\Om'\setminus \Om}\C(x/\e) \nabla w\cdot \nabla w\,dx
\\[1mm]=\int_\Om W_\e(x,\nabla u)\,dx+\int_{\Om} H\left( \frac x\e,\frac{D^su}{|D^su|}\right)\,d|D^su|\\
+\int_{\Gamma_d}H\left(\frac x\e, (w-u)\nu\right)\,d\hn+\frac{1}{2} \int_{\Om'\setminus \Om}\C(x/\e) \nabla w\cdot \nabla w\,dx.
\end{multline*}
Let us extend it to $L^1(\Om')$ by setting $\fs_\e(u)=+\infty$ for $u\in L^1(\Om')\setminus BV(\Om')$. 

The results of Section \ref{sec:twoscale} imply the following 

\begin{theorem}
\label{thm:FshomEshom}
The sequence $(\fs_\e)$ $\Gamma$-converges with respect to the strong topology of $L^1(\Om')$ to a functional $\fs^{hom}$ such that,   for $u\in BV(\Om')$ with $u=w$ on $\Om'\setminus \Omb$,
$$
\fs^{hom}(u)=\min_{(e,p)} \left\{\es^{hom}(u,e,p)\,:\, (u,e,p)\in \as(w,\Om')\right\},
$$
where $\es^{hom}$ is  defined in \eqref{eq:Ehom}.
\end{theorem}

\begin{proof}
Note that the minimum in the statement is attained. Indeed, considering a minimizing sequence $(e_n,p_n)$, an associated $(E_n,P_n)$ according to \eqref{eq:Ehom} is bounded in $L^2(\Om'\times \ys;\R^N)\times \ms_b(\Om'\times \ys;\R^N)$, so that $(e_n,p_n)$ is bounded in $L^2(\Om';\R^N)\times \ms_b(\Om';\R^N)$: then compactness under weak/weak* convergence, the lower semicontinuity properties of $\eshom$ together with Remark \ref{rem:unique-mu} yield the conclusion.
\par
By general $\Gamma$-convergence results, we may assume that $\fs^{hom}$ exists up to subsequences. Let us prove that
\begin{equation}
\label{eq:ineq1}
\min_{(e,p)} \left\{\es^{hom}(u,e,p)\,:\, (u,e,p)\in \as(w,\Om')\right\}\le \fs^{hom}(u).
\end{equation}
We may assume $\fs^{hom}(u)<+\infty$. Let $u_\e\in BV(\Om')$ with $u_\e=w$ on $\Om'\setminus \Omb$ be such that $u_\e\to u$ strongly in $L^1(\Om')$ and $\fs_\e(u_\e)\to \fs^{hom}(u)$. Let us consider $(u_\e,e_\e,p_\e)\in \as(w,\Om')$ such that
$$
\fs_\e(u_\e)=\es_\e(u_\e,e_\e,p_\e)
$$
according to \eqref{eq:oneFe}. Up to a subsequence we have 
$$
e_\e\weak e\qquad\text{weakly in }L^2(\Om';\R^N)
\qquad\text{and}\qquad p_\e\weakst p\qquad\text{weakly* in }\ms_b(\Om';\R^N)
$$
with $(u,e,p)\in \as(w,\Om')$, and, thanks to Theorem \ref{thm:hom-res}, 
\begin{multline*}
\fs^{hom}(u)=\lim_{\e\to0^+} \fs_\e(u_\e)=\lim_{\e\to0^+} \es_\e(u_\e,e_\e,p_\e)\\
 \ge \es^{hom}(u,e,p)\ge \min_{(e,p)} \left\{\es^{hom}(u,e,p)\,:\, (u,e,p)\in \as(w,\Om')\right\},
\end{multline*}
so that \eqref{eq:ineq1} follows.
\par
 Let us now prove that
\begin{equation}
\label{eq:ineq2}
\fs^{hom}(u)\le \min_{(e,p)} \left\{\es^{hom}(u,e,p)\,:\, (u,e,p)\in \as(w,\Om')\right\}.
\end{equation}
If $(u,e,p)\in \as(w,\Om')$, in view of Theorem \ref{thm:hom-res} there exists $(u_\e,e_\e,p_\e)\in \as(w,\Om')$ such that
$(u_\e,e_\e,p_\e)\weak (u,e,p)$ and
$$
\es^{hom}(u,e,p)=\lim_{\e\to0^+} \es_\e(u_\e,e_\e,p_\e).
$$
We can thus write
$$
\es^{hom}(u,e,p)=\lim_{\e\to0^+} \es_\e(u_\e,e_\e,p_\e)\ge \liminf_{\e\to0^+} \fs_\e(u_\e)\ge \fs^{hom}(u)
$$
and \eqref{eq:ineq2} follows. 
\par
$\fs^{hom}$ is thus uniquely defined and  there is no need to pass to subsequences.
\end{proof}

\begin{remark}\label{rmk:FF} In view of Theorem \ref{thm:hom-res} an alternative characterization of $\fs^{hom}(u), u\in BV(\Om')$ with $u=w$ on $\Om'\setminus \Omb$, is
\bel{eq:charfhom}
\fs^{hom}(u)= \min_{(E,P)}\{\eshom(u,E,P): (u,E,P)\in \ashom(w,\Om')\}.
\ee
This result demonstrates that there is no gap, in the periodic setting, between the two-scale viewpoint and the one-scale viewpoint. 

A  result of similar ilk has been obtained in \cite[Theorem 1.1]{FF} for a functional defined on $BV$. That result does not possess  the kinematics  of \eqref{set:Aw2} which are specific to elasto-plasticity and operates under a continuity assumption which does not allow the consideration of a two-phase setting. Specifically, it is in particular assumed  there that
$$
y\mapsto \mathbb C(y) \mbox{ and } y\multimap K(y) \mbox{ are continuous (multi)-functions}
$$
(see e.g. \cite[(2.10), (2.11)]{francfort.giacomini} for the definition of a continuous multifunction) in lieu of \eqref{eq:Cpos2}, \eqref{eq:defHper}. Then, upon defining
$
\bs_{2s}^{hom}(w,\Om')
$
as $\as_{2s}^{hom}(w,\Om')$, but with the simpler kinematics $M=Du\otimes\leb_y^N+D_y\mu$ where $M\in  \ms_b(\Om'\times \ys;\R^N)$ and $\mu\in \ms_b(\Om'; BV(\ys))$,
$$
\fs^{hom}(u)= \min_{M}\left\{\int_{\Om'\times\ys}W\left(y, \frac{M^{ac}}{\leb^{2N}}\right)dxdy+ \int_{{\Om'\times\ys}}\!\!H\left(y,\frac{M^{s}}{|M^s|}\right)d|M^s|  : (u,M)\in \bs_{2s}^{hom}(w,\Om')\right\}.
$$
If adopting that standpoint there is no longer any remembrance of the elasto-plastic structure of the density $W(y,\xi)$ in contrast with \eqref{eq:charfhom}.
\hfill\P\end{remark}

According to an early work of Bouchitt\'e \cite[Theorem 4.1  and Remark 4.2]{bouchitte86} which, in the text, requires $C^1$-regularity for the domain but, in fact, also applies to  Lipschitz domains, 
the functionals $\fs_\e$ $\Gamma$-converge in the strong topology of  $L^1(\Om')$ to 
\begin{multline}
\label{eq:ofgl}
\fs^{hom}(u)=\int_{\Om} W^{hom}(\nabla u)\,dx+\int_{\Om'} (W^{hom})^\infty\left(\frac{D^su}{|D^su|}\right)\,d|D^su|+\frac{1}{2} \int_{\Om'\setminus \Om}\bar \C \nabla w\cdot \nabla w\,dx\\[1mm]
=\int_\Om W^{hom}(\nabla u)\,dx+\int_{\Om} (W^{hom})^\infty\left(\frac{D^su}{|D^su|}\right)\,d|D^su|\\
+\int_{\Gamma_d}(W^{hom})^\infty((w-u)\nu)\,d\hn+\frac{1}{2} \int_{\Om'\setminus \Om}\bar \C \nabla w\cdot \nabla w\,dx
\end{multline}
for $u\in BV(\Om')$ with $u=w$ on $\Om'\setminus \Omb$, and $\fs^{hom}=+\infty$ otherwise,  where $\bar \C=\int_\ys \C\,dy$. In \eqref{eq:ofgl} 
\begin{multline}\label{eq:whom}
W^{hom}(\xi)=\inf_\ph\left\{\int_\ys W(y,\xi+\nabla_y\ph)\ dy: \ph\in W^{1,1}(\ys)\right\}\\[1mm]
=\inf_\ph\left\{\int_Y W(y,\xi+\nabla_y\ph)\ dy: \ph\in W^{1,1}(Y) \mbox{ with } \ph \;Y-\text{periodic}\right\}.
\end{multline}
\par

Obviously, we have not used the integral representation \eqref{eq:ofgl}, \eqref{eq:whom} in our previous characterization \eqref{eq:charfhom} of $\fs^{hom}$. In the final subsection we propose to directly investigate  the cell formula \eqref{eq:whom} so as to adjudicate the elasto-plastic character of the homogenized functional $\fs^{hom}$. We do so with some trepidation, having  failed to make use of Theorem \ref{thm:hom-res}  
in such an endeavor.

\begin{remark}\label{rem:DQ} In the vectorial setting alluded to in the introduction the energy associated with the reduced one-field formulation    exhibits quadratic growth along hydrostatic matrices (multiples of the identity) while it only has linear growth along deviatoric matrices.  In that setting a result similar to that of \cite{bouchitte86} was obtained in \cite[Equation (3.1)]{DQ}. To the best of our knowledge that work did not give rise to further explorations of the elasto-plastic character of the accompanying energy density. 
\hfill\P\end{remark}

\subsection{The elasto-plastic character of the homogenized energy}\label{sub:cell-form}
We can relax the energy in \eqref{eq:whom} and rewrite the homogenized density as
 \bel{eq:cell-pbm1}
 W^{hom}(\xi)= \min\left\{\int_\ys W(y, \xi+\nabla_y\ph) dy+ \int_\ys  W^{\infty} \left(y,\frac{d D^s_y\ph}{d|D^s_y\ph|}\right)d|D^s_y\ph|:\ph\in BV(\ys)\right\}.
 \ee
where the infimum is now a minimum because the problem is coercive in $BV(\ys)$. Indeed, because  the torus has no boundary, the limsup inequality can be proved directly as in Section \ref{subsec:2s} employing  localization and translation near the interface $\Sigma$ then convolution in $y$. On each phase $W$ is constant in $y$ and  the standard property of convex functions (with linear growth) of a measure under convolution (see \cite[Theorem 4]{GS}) permits to conclude. The liminf inequality is local and is thus a consequence of \cite{ADF} as seen in the proof of Lemma \ref{rem:fs-rel}. 
Note that a direct proof which would completely bypass the results of \cite{ADF} could also be devised using localization and partitions of unity. We will not dwell any further on that point.
\par
 Now  recalling the definition of $W$ in \eqref{eq:eq-infs}-\eqref{eq:defWx} as well as  equality \eqref{eq:H=Winf} and defining 
 $$
\qs_y(e):=\frac{1}{2}\int_\ys \C(y)e(y)\cdot e(y) dy, \quad
\hs_y(p):=\int_\ys H\left(y,\frac{p}{|p|}(y)\right)d|p|
 $$
for $e\in L^2(\ys,\R^N)$ and $p\in \ms_b(\ys;\R^N)$,  the equivalence between the elasto-plastic and the one-field approaches expounded in Section \ref{subsec:one-filed} (and now performed  in $BV(\ys)$) lead to the following
 \begin{proposition}\label{eq:cell-pbm2} The energy ${W}^{hom}$ in \eqref{eq:cell-pbm1} also reads as
 \begin{multline*}
 {W}^{hom}(\xi)= \min\left\{\qs_y(e)+\hs_y(p): (\ph,e,p)\in BV(\ys)\times L^2(\ys,\R^N)\times\ms_b(\ys;\R^N),\right.\\[2mm]\left.\xi+D_y\ph=e+p\right\}.
 \end{multline*}
 \end{proposition}
 
 If $\es^{hom}(u,e,p)$ was an elasto-plastic functional, the density $W^{hom}$ of $\fs^{hom}$ should present an inf-convolution structure involving a quadratic elastic energy and a convex positively one homogeneous dissipation plastic potential. This seems to be suggested by Proposition \ref{eq:cell-pbm2} in a variational, rather than  pointwise sense. The analysis of the cell formula is  non trivial. 
 In the rest of this study  we will address the one dimensional case, for which we can exhibit a positive answer, and  the multi-dimensional case  for which we will only be in a position to exhibit  bounds for $W^{hom}(\xi)$ in terms of elasto-plastic energy densities.

\subsubsection{The 1d-setting}\label{subsub:1d}
Consider a two-phase setting where the characteristic function of phase 1 (say $[0,\theta], 0<\theta<1),$ is denoted by $\chi_1(y)$,   the elasticities are respectively $\alpha,\beta$, both positive, and the yield stresses are $\kappa,\eta$ with $\kappa\le\eta$ by which we mean
$K_1=[-\kappa,\kappa],\; K_2=[-\eta,\eta]$. According to  Proposition \ref{eq:cell-pbm2}, for any $\xi\in\R$,
\begin{multline}\label{eq:1dhe}
W^{hom}(\xi)=\min_{(\varphi,e,p)}\bigg\{\frac{1}{2}\int_{(0,1)}[\chi_1(y)\alpha+(1-\chi_1(y))\beta]\,|e(y)|^2\ dy \\[1mm]+\int_{]0,1]}[\chi_1(y)\kappa+(1-\chi_1(y))\eta]\,d|p|: \varphi\in  BV(0,1),\;\xi+ \varphi'=e+p,\\[1mm] p(1)= \varphi^+- \varphi^- (\mbox{where  $\varphi^\pm$ are the traces of $u$ at $x=1$ resp. $0$})  \bigg\}.
\end{multline}
 Note that we took care to define $\chi_1(y)$ to equal $1$ on $[0,\theta]$ and not only on $[0,\theta[$ to abide by condition \eqref{eq:defHper}. Also note that $p$ should be defined on $]0,1]$ and not on $[0,1]$ so as to reflect the underlying periodicity of the relevant fields.
\par
In the 1d case, this minimum can be explicitly computed thanks to the associated Euler-Lagrange equations. These can be derived by adapting line by line the arguments at the beginning of Section 3 in \cite{BF} (see in particular \cite[Equation (3.2) and Definition 2.1]{BF}) to the one dimensional periodic $BV$ setting (so without boundary): the discontinuous coefficients are no obstacle because only  the quadratic nature of the elastic energy and the subadditive and convex character of the plastic dissipation are used there. We obtain
\bel{eq:ELeq}
\begin{cases}
\sigma(y)=[\chi_1(y) \alpha+(1-\chi_1(y))\beta]\,e(y)\\[2mm]
\sigma(y)=\sigma_0\in\R\\[2mm]
|\sigma_0|\le \chi_1(y) \kappa+(1-\chi_1(y))\eta \quad \text{a.e. in }  [0,1]\\[2mm]
\sigma_0 p= [\chi_1(y) \kappa+(1-\chi_1(y))\eta]\,|p|.
\end{cases}
\ee
Since $|\sigma_0|\le \kappa$  we conclude  from the last equation in \eqref{eq:ELeq}  that, either $p\equiv 0$, or $\sigma_0=\pm\kappa$ and $p=\pm |p|$ with $\text{supp}|p|\subset ]0,\theta]$.
 
 From $\xi+\varphi'=e+p$ we get that, if $p\equiv 0$ which implies that $\varphi^+=\varphi^-$, 
 \bel{eq:trans-11}\xi=\int_{]0,1[}e(y) dy=\left(\frac\theta\alpha+\frac{(1-\theta)}\beta\right)\sigma_0.
 \ee
 Otherwise, 
 $$
 \xi+\varphi'=\pm \left(\frac{1}{\alpha}\chi_1(y)+\frac{1}{\beta}(1-\chi_1(y))\right)\kappa\pm|p|,
 $$
or still, integrating over $[0,1]$,
 \bel{eq:trans-12}\xi=\pm \left(\frac\theta\alpha+\frac{(1-\theta)}\beta\right)\kappa\pm\int_{]0,\theta]}d|p|.
 \ee
 From \eqref{eq:trans-11}, we have that
 $$
 |\xi|<\left(\frac\theta\alpha+\frac{(1-\theta)}\beta\right)\kappa\Rightarrow p=0\Rightarrow e(y)=\left(\frac{\chi_1(y)}\alpha+\frac{(1-\chi_1(y))}\beta\right)\frac\xi{\displaystyle\left(\frac\theta\alpha+\frac{(1-\theta)}\beta\right)},
 $$
 and then, for such $\xi$'s,
 \bel{eq:trans-13}
 W^{hom}(\xi)=\frac{1}{2}\frac{|\xi|^2}{\displaystyle\left(\frac\theta\alpha+\frac{(1-\theta)}\beta\right)}.
 \ee
 Meanwhile, from \eqref{eq:trans-12},
 $$
 |\xi|\ge\left(\frac\theta\alpha+\frac{(1-\theta)}\beta\right)\kappa\Rightarrow \sigma_0=\pm\kappa\Rightarrow e(y)=\pm\left(\frac{\chi_1(y)}\alpha+\frac{(1-\chi_1(y))}\beta\right)\kappa,
 $$
 and then, for such $\xi$'s,
 \bel{eq:trans-14}
 W^{hom}(\xi)=\frac{1}{2}{\displaystyle\left(\frac\theta\alpha+\frac{(1-\theta)}\beta\right)}{\kappa^2}+\kappa\left|\xi\mp\left(\frac\theta\alpha+\frac{(1-\theta)}\beta\right) \kappa\right|.
 \ee
 Collecting \eqref{eq:trans-13}, \eqref{eq:trans-14}, we obtain the following
 \begin{proposition}\label{prop:1dhhom}
 In a one-dimensional setting, the homogenized energy $W^{hom}$ of a bi-phase material for which the first phase has volume fraction $\theta$ is that  associated with a homogeneous Hencky elasto-plastic problem with, for elasticity, the harmonic mean $\left(\displaystyle\frac\theta\alpha+\frac{(1-\theta)}\beta\right)^{ -1}$ of the elasticities $\alpha,\beta$ of the two phases and, for ``yield stress" the minimum of the yield stresses $\kappa,\eta$ of the two phases.
 \end{proposition}
 
 We finally investigate the multi-dimensional setting and produce bounds on $W^{hom}$ which, unfortunately, do not allow us to prove, or disprove the elasto-plastic character of $W^{hom}$ as was the case in one dimension.

\subsubsection{Bounds on the homogenized energy $W^{hom}$}\label{subsub:bounds}
We will restrict our attention to the case where 
\bel{eq:ci=i}
 \C_1=\alpha \II,\; \C_2=\beta \II \;(\alpha,\beta>0)
\ee
and also where, for some $\lambda>0$,
\bel{eq:ki=ball}
K_1={\overline B}(0,\lambda\alpha),\;  K_2={\overline B}(0,\lambda\beta).
\ee
This is of course a restrictive assumption.
\par
If $\xi\in\R^N\mapsto F(\xi)\ge 0$ is a  convex function  whose Legendre transform $F^*$ satisfies $F^*(\xi)=+\infty$ as $|\xi|\ge R$, 
we set
$$
F^{hom}(\xi):=\inf_\ph  \left\{\int_\ys(\chi_1(y)\alpha+(1-\chi_1(y))\beta) F(\xi+\nabla \ph)\ dx: \ph\in W^{1,1}( \ys)\right\}
$$
where $\chi_1(y)$ denotes the characteristic function of, say, phase 1, with $\int_\ys \chi_1(y) dy=\theta$. Then, 
observing that, for $t>0$, $(tF)^*(\xi)=tF^*(\xi/t)$, and using Fan's mini-max theorem (see e.g. \cite[Theorem 3.8]{ambrosio-tr}, with $\tau \in L^2(\ys;\R^N)$ with $|\tau| \le R$ under the weak topology) 
\begin{multline}\label{eq:exp-f*}
F^{hom}(\xi)\!=\displaystyle\inf_\ph \sup_{\tau(y)}\left\{\int_\ys \left[(\xi+\nabla \ph)\cdot \tau\!-\!(\chi_1(y)\alpha+(1\!-\!\chi_1(y))\beta)F^*\!\!\left(\frac{\tau}{\chi_1(y)\alpha+(1-\chi_1(y))\beta}\right)\right] \!dy\!\right\}\\[2mm]
=\displaystyle\sup_{\tau(y)}\inf_\ph \left\{\int_\ys \left[(\xi+\nabla \ph)\cdot \tau- (\chi_1(y)\alpha+(1\!-\chi_1(y))\beta)F^*\left(\frac{\tau}{\chi_1(y)\alpha+(1-\chi_1(y))\beta}\right)\right]\ dy\right\}\\[2mm]
=\sup_{\tau(y)}\left\{T\cdot\xi-\int_\ys \left[(\chi_1(y)\alpha+(1-\chi_1(y))\beta)F^*\left(\frac{\tau}{\chi_1(y)\alpha+(1-\chi_1(y))\beta}\right)\right]\ dy: \right.\\[2mm]\left.\tau \text{ with } {\rm div\;}\tau=0, T=\int_\ys\tau(y) dy\right\}.
\end{multline}
With this in mind, we proceed with the computation of $W^{hom}$ under assumptions \eqref{eq:ci=i}, \eqref{eq:ki=ball}.
It is immediately seen that
$$
W(y,\xi)=(\chi_1(y)\alpha+(1-\chi_1(y)\beta)\hat W(\xi)
$$
with
$$
\hat W(\xi)=\min_{\eta\in\R^N}\left\{ \frac{1}{2}|\xi-\eta|^2+\lambda |\eta|\right\}.
$$
In other words $\hat W$ can be thought of as corresponding to an elasto-plastic material with elasticity $\II$ and yield stress $\lambda$ (i.e., $K={\overline B}(0,\lambda)$). Recalling the last equality in \eqref{eq:exp-f*}, we get that
\begin{multline*}
W^{hom}(\xi)=\sup_{\tau(y)}\left\{T\cdot\xi-\int_\ys \left[(\chi_1(y)\alpha+(1-\chi_1(y))\beta)(\hat W)^*\left(\frac{\tau}{\chi_1(y)\alpha+(1-\chi_1(y))\beta}\right)\right]\ dy: \right.\\[1mm]\left.\tau \text{ with } {\rm div\;}\tau=0, T=\int_\ys\tau(y) dy\right\}.
\end{multline*}
But $(\hat W)^*(\xi)= \frac{1}{2}|\xi|^2+\mathbf I_{{\overline B}(0,\lambda)}(\xi).$ Consequently,
\begin{multline}\label{eq:compwhom}
W^{hom}(\xi)=\sup_{\tau(y)}\left\{T\cdot\xi-\int_\ys \left[ \frac{1}{2}\left(\frac{\chi_1(y)}\alpha+\frac{(1-\chi_1(y))}\beta\right)|\tau|^2+\right.\right.\\[1mm]\left.\left.{(\chi_1(y)\alpha+(1\!-\!\chi_1(y))\beta)}\mathbf I_{{\overline B}(0,\lambda)}\left(\frac\tau{\chi_1(y)\alpha+(1-\chi_1(y))\beta}\right)\right]\ dy: \tau \text{ with } {\rm div\;}\tau=0, T=\!\int_\ys\tau(y) dy\right\}\\[2mm]
=\sup_{T\in\R^N}\left\{T\cdot\xi- \frac{1}{2}\inf_{\tau(y)}\left[\int_\ys \left(\frac{\chi_1(y)}\alpha+\frac{(1-\chi_1(y))}\beta\right)|\tau|^2\ dy\right]:\tau \text{ with } {\rm div\;}\tau=0, \right.\\[1mm]\left. |\tau(y)|\le \lambda (\chi_1(y)\alpha+(1-\chi_1(y))\beta),\;T=\!\int_\ys\tau(y) dy\right\}.
\end{multline}

We can specialize the $\tau$'s to be constant fields ($\tau=T$). Then \eqref{eq:compwhom} becomes
\begin{multline*}
W^{hom}(\xi)\ge \sup_{T\in\R^N}\left\{T\cdot\xi-\frac{1}{2}\left(\frac{\theta}\alpha+\frac{(1-\theta)}\beta\right)|T|^2: |T|\le \min{(\lambda \alpha,\lambda\beta)}\right\}\\[1mm]
=\min_{p\in\R^N}\left\{\frac{1}{2}\left(\frac{\theta}\alpha+\frac{(1-\theta)}\beta\right)^{-1}|\xi-p|^2+\min{(\lambda \alpha,\lambda\beta)}|p| \right\}.
\end{multline*}
In other words, defining $W_{min}$ to be the one-field energy associated with the elasto-plastic material characterized by
\bel{eq:wmin}
\C=\left(\frac{\theta}\alpha+\frac{(1-\theta)}\beta\right)^{ -1}\II,\;K={\overline B}(0, \min{(\lambda \alpha,\lambda\beta)},
\ee
 we get
$$
 W^{hom}(\xi)\ge W_{min}(\xi).
 $$
 We can also remark that by Jensen's inequality for  $(t,\xi)\mapsto |\xi|^2/t$, for every admissible $\tau$ we get
$$
 \int_\ys \left(\frac{\chi_1(y)}\alpha+\frac{(1-\chi_1(y))}\beta\right)|\tau|^2\ dy\ge \frac{|T|^2}{\theta\alpha+(1-\theta)\beta}.
$$
 So, since $|\tau(y)|\le \lambda (\chi_1(y)\alpha+(1-\chi_1(y))\beta)$ implies that $|T|\le {(\theta\lambda\alpha+(1-\theta) \lambda\beta)}$, recalling \eqref{eq:compwhom} once more, we obtain 
\begin{multline*}
W^{hom}(\xi)\le \sup_{T\in\R^N}\left\{T\cdot\xi-\frac{1}{2} \frac{|T|^2}{\theta\alpha+(1-\theta)\beta}: |T|\le {(\theta\lambda\alpha+(1-\theta) \lambda\beta)} \right\}\\[1mm]=\min_{p\in\R^N}\left\{\frac{1}{2}(\theta\alpha+(1-\theta)\beta)|\xi-p|^2+{(\theta\lambda\alpha+(1-\theta) \lambda\beta)}|p|\right\}.
\end{multline*}
In other words, defining $W_{max}$ to be the one-field energy associated with the elasto-plastic material characterized by
\bel{eq:wmax}
\C=\left({\theta}\alpha+{(1-\theta)}\beta\right)\II,\;K={\overline B}(0, {\theta\lambda \alpha+(1-\theta)\lambda\beta)},
\ee
 we get
 $$
 W^{hom}(\xi)\le W_{max}(\xi).
 $$
 We have thus proved the following
 \begin{proposition}\label{prop:bds-Whom} Assuming that \eqref{eq:ci=i}, \eqref{eq:ki=ball} hold true,
 the homogenized energy $W^{hom}$ satisfies
 $$
 W_{min}(\xi)\le W^{hom}(\xi)\le W_{max}(\xi)
 $$
 where $W_{min}, W_{max}$ defined through  \eqref{eq:wmin}, \eqref{eq:wmax} are associated with an elasto-plastic material.
 \end{proposition}
 
 \begin{remark} Unfortunately, the above proposition does not prove or disprove the conjecture that $W^{hom}$ is always -- at least in the case of a two-phase microstructure  further satisfying \eqref{eq:ci=i}, \eqref{eq:ki=ball} -- associated with an elasto-plastic behaviour. Summing up, we contend that the extent to which elasto-plastic behaviour is preserved by homogenization is still unknown. 
 
An example where this is so in the vectorial setting has been put forth in the Appendix of \cite{DQ}. In that example,    the associated energies lack  coercivity bounds of the type implied by \eqref{eq:Cinc} (the associated  $x$-independent elasticity is degenerate; see (A1) in that paper), so that the $\Gamma$-convergence result does not provide any information about the asymptotic behavior of possibly non-existing minimizers. Further, the choice of the yield conditions (the equivalent of the choice of the sets $K_i, i=1,2$) is ad hoc and does not fit  any standard yield criterion.
 \hfill\P\end{remark}

\section*{Appendix: distributions and measures on the periodic torus}\label{sec:appendix}

\setcounter{theorem}{0}
\setcounter{equation}{0}
\setcounter{section}{1}

\renewcommand{\thesection}{\Alph{section}}

In this Appendix we  offer a self-contained approach to the study of distributions, and in particular of measures, on the periodic torus $\ys:=\R^N/\Z^N$ ($\pi_\ys$ being the associated projection). The approach is based on the natural identification between functions on $\ys$ and $\Z^N$-periodic functions on $\R^N$. Recall that, for $\varphi$ function on $\ys$, we denote with $\varphi_{\R^N}$ the associated periodic function on $\R^N$ and conversely, given a periodic function $\psi$ on $\R^N$, we write $\psi_\ys$ for the associated function on $\ys$.
 \par
For $\psi\in \ds(\R^N)$ let us set
\begin{equation}
\label{def:psihat}
\hat \psi(x):=\sum_{j\in\Z^N}\psi(x+j).
\end{equation}
Notice that $\hat\psi$ is a smooth periodic function on $\R^N$, so that $\hat \psi_{\ys}\in \ds(\ys)$. For every $\Phi\in \ds'(\ys)$ ($\ds(\ys)$ is endowed with the usual Schwartz topology) we set
$$
\pscal{\Phi_{\R^N}}{\psi}:=\pscal{\Phi}{\hat \psi_{\ys}}.
$$

\begin{proposition}
\label{prop:D1}
The following items hold true.
\begin{itemize}
\item[(a)] $\Phi_{\R^N} \in \ds'(\R^N)$.
\item[(b)] $\Phi_{\R^N}$ is periodic, i.e., for every $j\in\Z^N$
$$
\pscal{\Phi_{\R^N}}{\tau_j\psi}=\pscal{\Phi_{\R^N}}{\psi},
$$
where $\tau_j\psi(x):=\psi(x-j)$.
\item[(c)] $D_y \Phi_{\R^N}=(D_y\Phi)_{\R^N}$.
\end{itemize}
\end{proposition}

\begin{proof}
If $K\subset\R^N$ is compact  and the support of $\psi$ is in $K$, then for every $x\in\R^N$ only a finite number $N_K$ of indices $j$ depending on $K$ are involved in the sum \eqref{def:psihat} defining $\hat\psi$ at $x$. 
If $\Phi$ has order $m$ 
$$
|\pscal{\Phi_{\R^N}}{\psi}|=|\pscal{\Phi}{\hat\psi_{\ys}}|\le C\max_{|\alpha|\le m}\|D^\alpha \hat\psi_{\ys}\|_\infty
=C\max_{|\alpha|\le m}\|D^\alpha \hat\psi\|_\infty \le CN_K \max_{|\alpha|\le m}\|D^\alpha\psi\|_\infty,
$$
that is $\Phi_{\R^N}$ is a distribution of order $m$ on $\R^N$. Point (a) is proved.
\par
 The periodicity of $\Phi_{\R^N}$ follows since 
$$
\pscal{\Phi_{\R^N}}{\tau_j\psi}=\pscal{\Phi_{\R^N}}{\widehat{\tau_j\psi}_\ys}=\pscal{\Phi_{\R^N}}{\widehat{\psi}_\ys}=\pscal{\Phi_{\R^N}}{\psi}
$$
as $\widehat{\tau_j\psi}=\hat \psi$.
\par
Concerning point (c) we can write
$$
\pscal{D_y\Phi_{\R^N}}{\psi}=-\pscal{\Phi_{\R^N}}{D_y\psi}=-\pscal{\Phi}{\widehat{D_y\psi}_\ys}=-\pscal{\Phi}{D_y\widehat{\psi}_\ys}=\pscal{D_y\Phi}{\widehat{\psi}_\ys}=\pscal{(D_y\Phi)_{\R^N}}{\psi},
$$
and the result follows.
\end{proof}

\begin{remark}
\label{rem:cube1}
If $\varphi\in \ds(\ys)$ is supported in $\pi_\ys(Q)$ with $Q$ an open cube of side less than $1$, then
$$
\pscal{\Phi_{\R^N}}{(\varphi_{\R^N})\lfloor{Q}}=\pscal{\Phi}{\varphi}
$$
where $(\varphi_{\R^N})\lfloor{Q}$ can be seen as an element of $\ds(Q)$. Indeed $\widehat{(\varphi_{\R^N})\lfloor{Q}}=\varphi_{\R^N}$ as
 the supports of $\tau_{j_1}\varphi$ and $\tau_{j_2}\varphi$ do not intersect   when $j_1\ne j_2$.
\hfill\P
\end{remark}

Regularization by convolution of distributions on $\ys$ can be naturally defined by considering the associated periodic distributions on $\R^N$ and performing the standard regularization. Specifically, given $\Phi\in \ds'(\ys)$ and a family of regularizing symmetric kernels $\{\rho_\e\}_{0<\e<1}$ on $\R^N$, we set
$$
\Phi*\rho_\e^y:= (\Phi_{\R^N}*\rho_\e)_{\ys}\in \ds(\ys).
$$
The definition is meaningful since $\Phi_{\R^N}*\rho_\e$ is a periodic smooth function on $\R^N$ because $\Phi_{\R^N}$ is periodic and
$$
(\Phi_{\R^N}*\rho_\e)(x)=\pscal{\Phi_{\R^N}}{\rho_\e(x-\cdot)}
$$

\begin{proposition}
\label{prop:conv-ys}
The following items hold true.
\begin{itemize}
\item[(a)] $D_y(\Phi*\rho_\e^y)=D_y\Phi*\rho_\e^y$.
\vskip5pt
\item[(b)] $\Phi*\rho_\e^y\weakst \Phi$ weakly* in $\ds'(\ys)$ as $\e\to 0^+$.
\end{itemize}
\end{proposition}

\begin{proof}
Let $f_\e:=\Phi_{\R^N}*\rho_\e$ and let $\varphi$ be supported in $\pi_\ys(Q)$ with $Q\subset \R^N$ an open cube with side less than $1-\e$. Then by definition of the distribution associated to a function, denoting with $Y_Q$ a cube with the same center as $Q$ and side 1, 
\begin{equation}
\label{eq:1conv}
\pscal{\Phi*\rho_\e^y}{\varphi}=\int_Y f_\e \varphi_{\R^N}\,dx=\int_{Y_Q}f_\e \varphi_{\R^N}\,dx=\pscal{\Phi_{\R^N}}{(\varphi_{\R^N}*\rho_\e)\lfloor{Q}}.
\end{equation}
As a consequence, taking into account Remark \ref{rem:cube1}, we have 
\begin{multline*}
\pscal{D_y(\Phi*\rho_\e^y)}{\varphi}=-\pscal{\Phi_{\R^N}}{((D_y\varphi)_{\R^N}*\rho_\e)\lfloor{Q}}
=-\pscal{\Phi_{\R^N}}{D_y((\varphi_{\R^N})\lfloor{Q}*\rho_\e)}\\
=\pscal{D_y\Phi_{\R^N}}{(\varphi_{\R^N})\lfloor{Q}*\rho_\e}=\pscal{(D_y\Phi)_{\R^N}}{(\varphi_{\R^N})\lfloor{Q}*\rho_\e}\\=\pscal{(D_y\Phi)_{\R^N}*\rho_\e}{(\varphi_{\R^N})\lfloor{Q}}=\pscal{D_y\Phi*\rho_\e^y}{\varphi}.
\end{multline*}
Since every function $\varphi\in \ds(\ys)$ can be seen as a sum of functions of the type above through a partition of unity, point (a) follows.
\par
Let us come to point (b). In view of Remark \ref{rem:cube1} we deduce from \eqref{eq:1conv} that
\begin{equation}
\label{eq:lim-weak}
\lim_{\e\to 0^+}\pscal{\Phi*\rho_\e^y}{\varphi}=\pscal{\Phi_{\R^N}}{(\varphi_{\R^N})\lfloor{Q}}=\pscal{\Phi}{\varphi}.
\end{equation}
Once again, since every function $\varphi\in \ds(\ys)$ can be seen as a sum of functions of the type above, the desired weak* convergence follows. 
\end{proof}

In this paper we are mostly interested in measures on $\ys$ which can be viewed as distributions of order zero thanks to Riesz representation theorem. If $\mu\in \ms_b(\ys)$, then the associated periodic order-zero distribution $\mu_{\R^N}$ will be a periodic Radon measure (unbounded if not zero) on $\R^N$. Indeed,

\begin{proposition}[\bf Convolutions of measures]
\label{prop:meas-reg}
If $\mu\in \ms_b(\ys)$ then
$$
\mu*\rho_\e^y \weakst \mu \qquad\text{weakly* in $\ms_b(\ys)$}
$$
with 
\begin{equation}
\label{eq:strict}
\lim_{\e\to 0^+}|\mu*\rho_\e^y|(\ys)=|\mu|(\ys).
\end{equation}
\end{proposition}

\begin{proof}
Weak* convergence in $\ds'(\ys)$ has been proved in Proposition \ref{prop:conv-ys}. If we check \eqref{eq:strict}, then by classical arguments the convergence is also in the weak* topology of $\ms_b(\ys)$ since continuous functions can be approximated uniformly by smooth functions on $\ys$. It suffices to check that
$$
\limsup_{\e\to 0^+}|\mu*\rho_\e^y|(\ys)\le |\mu|(\ys),
$$
the inequality $|\mu|(\ys)\le \liminf_{\e\to0^+}|\mu*\rho_\e^y|(\ys)$ being standard. If $\varphi\in \ds(\ys)$ and if $\{\zeta_i\}_{i=1,\dots,k}$ is a partition of unity on $\ys$ subordinated to an open cover of the type $\{\pi_\ys(Q_i)\}_{i=1,\dots,k}$ with $Q_i$ open cube in $\R^N$ with side less than 1, we have for $\e$ small enough
$$
\pscal{\mu*\rho_\e^y}{\varphi}=
\sum_{i=1}^k\int_{Q_i}(\zeta_i\varphi)_{\R^N}*\rho_\e\,d\mu_{\R^N}.
$$
Now for $x\in Q_i$ we have
\begin{multline*}
[(\zeta_i\varphi)_{\R^N}*\rho_\e](x)-[(\zeta_i)_{\R^N} (\varphi_{\R^N}*\rho_\e)](x)\\
=\int_{Q_i}(\zeta_i)_{\R^N}(z) \varphi_{\R^N}(z) \rho_\e(x-z)\,dz-\int_{Q_i}(\zeta_i)_{\R^N}(x)\varphi_{\R^N}(z) \rho_\e(x-z)\,dz\\
=\int_{Q_i}[(\zeta_i)_{\R^N}(z)-(\zeta_i)_{\R^N}(x)] \varphi_{\R^N}(z)\rho_\e(x-z)\,dz
\end{multline*}
so that, for $\e$ small,
$$
\|(\zeta_i\varphi)_{\R^N}*\rho_\e-(\zeta_i)_{\R^N} (\varphi_{\R^N}*\rho_\e)\|_\infty\le C_i \e\|\varphi_{\R^N}\|_\infty,
$$
where $C_i$ is the Lipschitz constant of $(\zeta_i)_{\R^N}$.
We deduce that for some $C$ independent of $\e$
$$
\pscal{\mu*\rho_\e^y}{\varphi}\le \sum_{i=1}^k \int_{Q_i} (\zeta_i)_{\R^N} (\varphi_{\R^N}*\rho_\e) \,d\mu_{\R^N}+C\e \|\varphi\|_\infty
=\pscal{\mu}{(\varphi_{\R^N}*\rho_\e)_\ys}+C\e\|\varphi\|_\infty
$$
so that passing to the sup in $\varphi$ with $\|\varphi\|_\infty\le 1$ we get
$$
|\mu*\rho_\e^y|(\ys)\le |\mu|(\ys)+C\e.
$$
and the result follows by letting $\e\to0^+$.
\end{proof}

\section*{Statement and declarations:}

A.G. is a member of the Gruppo Nazionale per l'Analisi Matematica, la Probabilit\`a e le loro Applicazioni (GNAMPA) of the Istituto Nazionale di Alta Matematica (INdAM). 

\par  Data sharing not applicable to this article as no datasets were generated or analyzed during the current study.

\par We thank both referees for their insightful comments which contributed to a better version of this work.

\bibliographystyle{plain}

\end{document}